\documentclass[12pt,twoside]{article}
\usepackage[margin=0.9in]{geometry}
\usepackage{fancyhdr}
\usepackage{graphicx}
\usepackage{amssymb}
\usepackage{amsmath}
\usepackage{amsfonts}
\usepackage{theorem}
\usepackage{mathrsfs}
\usepackage{mathtools}
\usepackage{empheq}
\usepackage{bm}
\usepackage[page,header]{appendix}
\usepackage{color}
\usepackage{setspace}
\usepackage{epstopdf}
\usepackage{float}
\usepackage{cite}
\newcommand\eref[1]{(\ref{#1})}
\usepackage{cleveref}

\usepackage{booktabs} 
\usepackage{array} 
\usepackage{paralist} 
\usepackage{verbatim} 
\usepackage{subfigure} 

\newtheorem{theorem}{Theorem}[section]

{\theorembodyfont{\rmfamily}\newtheorem{remark}{Remark}[section]}
\newenvironment{proof}{{\flushleft \bf Proof:}}{}

\newenvironment{acknowledgment}{{\flushleft \bf Acknowledgment:}}{}

\allowdisplaybreaks[1]

\numberwithin{equation}{section}
\numberwithin{figure}{section}
\numberwithin{table}{section}

\usepackage{ae,aecompl}

\newcommand{\mx}{\bm{x}}

\newcommand{\mq}{\bm{q}}

\newcommand{\mf}{\bm{F}}
\newcommand{\ms}{\bm{S}}
\newcommand{\mg}{\bm{G}}

\newcommand{\mo}{\bm{0}}

\newcommand{\dx}{\Delta x}
\newcommand{\dy}{\Delta y}

\newcommand{\hf}{{\frac{1}{2}}}

\newcommand{\jph}{{j+\frac{1}{2}}}
\newcommand{\jmh}{{j-\frac{1}{2}}}
\newcommand{\jpmh}{{j\pm\frac{1}{2}}}
\newcommand{\kpmh}{{k\pm\frac{1}{2}}}
\newcommand{\kph}{{k+\frac{1}{2}}}
\newcommand{\kmh}{{k-\frac{1}{2}}}

\def\softd{{\leavevmode\setbox1=\hbox{d}%
          \hbox to 1.05\wd1{d\kern-0.4ex{\char039}\hss}}}

\def\Chi {{\Large{\mbox{$\chi$}}}}

\newcommand*\xbar[1]{%
  \hbox{%
    \vbox{%
      \hrule height 0.5pt 
      \kern0.5ex
      \hbox{%
        \kern-0.1em
        \ensuremath{#1}%
        \kern-0.0em
      }%
    }%
  }%
}
\makeatletter
\newcommand{\refcheckize}[1]{%
  \expandafter\let\csname @@\string#1\endcsname#1%
  \expandafter\DeclareRobustCommand\csname relax\string#1\endcsname[1]{%
    \csname @@\string#1\endcsname{##1}\wrtusdrf{##1}}%
  \expandafter\let\expandafter#1\csname relax\string#1\endcsname
}
\makeatother

\title{Well-Balanced Schemes for the Euler Equations with Gravitation: Conservative Formulation Using Global Fluxes}
\author{Alina Chertock\thanks{Department of Mathematics, North Carolina State University, Raleigh, NC 27695, USA;
{\tt chertock@math.ncsu.edu}}~{}, 
Shumo Cui\thanks{Department of Mathematics, Temple University, Philadelphia, PA 19122, USA; {\tt shumo.cui@temple.edu}}~{},
Alexander Kurganov\thanks{Department of Mathematics, Southern University of Science and Technology of China, Shenzhen, 518055, China and
Mathematics Department, Tulane University, New Orleans, LA 70118, USA; {\tt kurganov@math.tulane.edu}}~{},\\
\c{S}eyma Nur \"{O}zcan\thanks{Department of Mathematics, North Carolina State University, Raleigh, NC 27695, USA; {\tt
snozcan@ncsu.edu}}~{} and Eitan Tadmor\thanks{Department of Mathematics, Center of Scientific Computation and Mathematical Modeling
(CSCAMM), Institute for Physical sciences and Technology (IPST), University of Maryland, College Park, MD 20742, USA; {\tt
tadmor@cscamm.umd.edu}}}
\date{}

\begin{document}

\maketitle
\begin{abstract}
We develop a second-order well-balanced central-upwind scheme for the compressible Euler equations with gravitational source term. Here, we
advocate a new paradigm based on a purely conservative reformulation of the equations using global fluxes. The proposed scheme is capable of
exactly preserving steady-state solutions expressed in terms of a nonlocal equilibrium variable. A crucial step in the construction of the
second-order scheme is a well-balanced piecewise linear reconstruction of equilibrium variables combined with a well-balanced
central-upwind evolution in time, which is adapted to reduce the amount of numerical viscosity when the flow is at (near) steady-state
regime. We show the performance of our newly developed central-upwind scheme and demonstrate importance of perfect balance between the 
fluxes and gravitational forces in a series of one- and two-dimensional examples.
\end{abstract}




\section{Introduction}
\newcommand{\bx}{{\bm{x}}}
\newcommand{\bu}{{\bm{u}}}

We are interested in approximations of nonlinear hyperbolic systems of balance laws, 
\begin{equation}
\mq_t+\nabla_{\bm{x}}\cdot\vec{\bm{F}}(\mq)=\ms(\mq), 
\label{sys}
\end{equation}
where $\bm{x}\in\mathbb{R}^d$ is a multivariate space variable, $t$ is the time, $\vec{\mf}(\mq)$ is a flux, and $\ms(\mq)$ is a source
term. Our main concern is with \emph{well-balanced} approximations for such systems, which have to address two competing aspects of 
\eqref{sys}. On one hand, they employ conservative discretization of the nonlinear flux to produce high-resolution approximation of the
transient solution, $\mq=\mq(\mx,t)$. On the other hand, a well-balanced scheme is expected to capture the correct steady-state solutions, 
$\mq_\infty=\mq_\infty(\mx)$, satisfying $\nabla_{\bm{x}}\cdot\vec{\bm{F}}(\mq_\infty)=\ms(\mq_\infty)$. 

To make our ideas concrete, we focus on the compressible Euler equations with gravitation, which in the two-dimensional (2-D) case reads as
\begin{equation}
\left\{\begin{aligned}
&\rho_t+(\rho u)_x+(\rho u)_y=0,\\
&(\rho u)_t+(\rho u^2+p)_x+(\rho uv)_y=-\rho\phi_x,\\
&(\rho u)_t+(\rho uv)_x+(\rho v^2+p)_y=-\rho\phi_y,\\
&E_t+(u(E+p)_x+(v(E+p))_y=-\rho u\phi_x-\rho v\phi_y,\\
\end{aligned}\right.
\label{1.2}
\end{equation}
where $\rho$ is the density, $u$ and $v$ are the velocities in the $x$- and $y$-directions, respectively, $E$ is the total energy,
$p:=(\gamma-1)(E-\hf \rho(u^2+v^2))$ is the pressure, and $\phi$ is the time-independent continuous gravitational potential. 

The system of balance laws \eref{1.2} is used to model astrophysical and atmospheric phenomena including supernova explosions 
\cite{KM14,KM16}, (solar) climate modeling and weather forecasting \cite{BKLL}. In many physical applications, solutions of this system are
sought as small perturbations of the underlying steady states. Capturing such solutions numerically is a challenging task since the size of
these perturbations may be smaller than the size of the truncation error affordable by the given computational  grid. It is important 
therefore, to design a \emph{well-balanced} numerical method, which is capable of exactly preserving (certain)  steady-state solutions, so
that the perturbed solutions will be resolved on the given grid, free of non-physical spurious oscillations.

\subsection{Conservative Formulation Using Global Fluxes}
Well-balanced schemes were introduced in \cite{GL} and mainly developed in the context of shallow water equations; see, e.g.,
\cite{ABBKP,CDKL,FMT11,GPS,Jin,KPshw,LeV98,NXS09,PS,RB09,XSN} and references therein. Further extensions to Euler equations with gravitation
will be mentioned below. Here, we advocate a new approach for designing well-balanced schemes for the system  \eref{1.2}. Our starting
point is to rewrite this system by incorporating the gravitational source terms into the corresponding momentum and energy fluxes, thus
arriving at the following \emph{purely conservative formulation}:
\begin{equation}
\left\{\begin{aligned}
&\rho_t+(\rho u)_x+(\rho v)_y=0,\\
&(\rho u)_t+(\rho u+K)_x+(\rho uv)_y=0,\\
&(\rho v)_t+(\rho uv)_x+(\rho v^2+L)_y=0,\\
&(E+\rho\phi)_t+(u(E+\rho\phi+p))_x+(v(E+\rho\phi+p))_y=0.
\end{aligned}\right.
\label{eq:global}
\end{equation}
Here, $K:=p+Q$ and $L:=p+R$, where $Q$ and $R$ involve the \emph{global variables} 
\begin{equation}
Q(x,y,t):=\int\limits^{x}\rho(\xi,y,t)\phi_x(\xi,y)\,d\xi,\quad R(x,y,t):=\int\limits^y\rho(x,\eta,t)\phi_y(x,\eta)\,d\eta.
\label{QR}
\end{equation}

It is the global momentum fluxes in \eqref{eq:global}, \eqref{QR} that play a central role in our proposed approach. We argue below, both 
theoretically and demonstrate numerically, that the use of such purely conservative formulation \eqref{eq:global}, \eqref{QR} is an 
effective approach to capture {\color{black}arbitrary (without any assumption of a thermal equilibrium)} motionless steady states given by
\begin{equation}
u\equiv0,\quad v\equiv0,\quad K_x\equiv0, \quad L_y\equiv0. 
\label{1.11}
\end{equation}
To this end, one needs to carefully integrate the global variables $Q$ and $R$ on the fly, to recover the reformulated momentum equations 
in terms of the corresponding global fluxes on the left-hand side of the second and third equations in \eref{eq:global}. The use of such 
global fluxes is not easily amenable, however, for numerical approximations which make use of (approximate) Riemann problem solvers. 
Instead, our well-balanced approach employs the class of \emph{central-upwind} (CU) schemes, \cite{KTcl,KLin,KNP,KTrp}, which offer highly 
accurate Godunov-type finite-volume methods that do \emph{not} require any (approximate) Riemann problem solver, and as such are natural 
candidates for numerical approximation of global-flux based formulation of the Euler equations \eref{eq:global}. The adaptation of the CU
schemes in the present context of global flux must be handled with care: in \S\ref{sec:lack} we show that a high-order reconstruction of
the \emph{conservative} variables $(\rho,\rho u,\rho v,E)^\top$ do not possess the well-balanced property. Instead, we introduce a special
reconstruction based on the \emph{equilibrium variables}, $(\rho,\rho u,\rho v,K,L)^\top$ rather than the conservative ones. The resulting
well-balanced CU schemes in the one- and two-dimensional cases are developed in \S\ref{sec2} and \S\ref{sec3}, respectively. In
\S\ref{sec4}, we present a series of examples of one- and two-dimensional numerical simulations which confirm the desired performance of our
proposed well-balanced, high-resolution CU scheme is coupled with the global-based fluxes of purely conservative formulation.

We close this introduction with a brief overview of related work on well-balanced schemes for the Euler equations with gravitational fields.
In \cite{LeVBale}, quasi-steady wave-propagation methods were developed for models with a static gravitational field. In \cite{BKLL}, 
well-balanced finite-volume methods, which preserve a certain class of steady states, were derived for nearly hydrostatic flows. The recent 
works \cite{ChaKli,DZBK16,TKK,KM16} introduce  finite-volume methods with carefully reconstructed solutions that handle more general 
gravitational potentials and preserve more general classes of steady states. In \cite{LXL,TXCD,XLC10}, gas-kinetic schemes were extended
to the multidimensional gas dynamic equations and well-balanced numerical methods were developed for problems, in which the gravitational
potential was modeled by a piecewise step function. More recently, higher order finite-difference \cite{XinShu13} and finite-volume 
\cite{LiXing} methods for the gas dynamics with gravitation have been introduced. 

In summary, we advocate the three main ingredients of the proposed approach: (i) the reformulation of Euler systems with gravitational 
source field as a purely conservative system using global-based fluxes; (ii) the use of Riemann-problem-solver free CU schemes to 
advance a highly-accurate numerical solution for such systems; and (iii) the reconstruction of such numerical solution using equilibrium 
rather conservative variables. 

\section{One-Dimensional Central-Upwind Scheme}\label{sec2}
In this section, we briefly describe the semi-discrete CU scheme applied to the one-dimensional (1-D) version of the system 
\eqref{eq:global} considered in the $y$-direction: 
\begin{equation}
\mq_t+\mg(\mq)_y=\mo,
\label{g11}
\end{equation}
where
\begin{equation}
\mq=(\rho,\rho v,E+\rho\phi)^\top,\quad\mg(\mq)=\left(\rho v,\rho v^2+L,v(E+\rho\phi+p)\right)^\top
\label{g12}
\end{equation}
and 
\begin{equation}
p=(\gamma-1)\Big(E-\frac{1}{2}\rho v^2\Big).
\label{eos1}
\end{equation}
The simplest steady state of \eref{g11}--\eref{eos1} is the motionless one for which 
\begin{equation}
v\equiv0,\quad L\equiv{\rm Const}.
\label{ss1}
\end{equation}
We assume that the cell averages of the numerical solution, $\xbar\mq_k(t):=\frac{1}{\dy}\int_{C_k}\mq(y,t)\,dy$, are available at time $t$.
Here, $\{C_k\}$ is a partition of the computational domain into finite-volume cells $C_k:=[y_\kmh,y_\kph]$ of size $|C_k|=\dy$ centered at
$y_k$, $~k=k_\ell,\ldots,k_r$. To advance the solution in time, we use the semi-discrete CU scheme \cite{KLin} applied to
\eref{g11}--\eref{eos1} which results in the following system of ODEs:
\begin{equation}
\frac{d}{dt}\,\xbar\mq_k=-\frac{\bm{\mathcal G}_\kph-\bm{\mathcal G}_\kmh}{\dy},
\label{2.1a}
\end{equation}
computed in terms of the CU numerical flux
\begin{equation}
\hspace*{-0.2cm}\bm{\mathcal G}_\kph:=\frac{b^+_\kph\mg(\mq_k^{\rm N})-b^-_\kph\mg(\mq_{k+1}^{\rm S})}{b^+_\kph-b^-_\kph}+
\beta_\kph\left(\mq_{k+1}^{\rm S}-\mq_k^{\rm N}-\delta\mq_\kph\right),~~\beta_\kph:=\frac{b^+_\kph b^-_\kph}{b^+_\kph-b^-_\kph}.
\label{2.2}
\end{equation}
Here, 
\begin{equation}
\mq_k^{\rm N}:=\widetilde{\mq}(y_\kph-0)=\xbar{\mq}_k+\frac{\dy}{2}{(\mq_y)}_k\quad\mbox{and}\quad
\mq_{k+1}^{\rm S}:=\widetilde{\mq}(y_\kph+0)=\xbar{\mq}_{k+1}-\frac{\dy}{2}{(\mq_y)}_{k+1}
\label{2.5}
\end{equation}
are the one-sided point values of the computed solution at the ``North'' and ``South'' cell interfaces $y=y_\kph$, which are recovered from 
a second-order  piecewise linear reconstruction
{\color{black}
$$
\widetilde{\mq}(y)=\sum_k\big(\xbar\mq_k+(\mq_y)_k(y-y_k)\big)\Chi_{C_k}(y),\quad\Chi_{C_k}(y)=
\left\{\begin{aligned}1,&~\mbox{if}~y\in C_k,\\0,&~\mbox{otherwise.}\end{aligned}\right.
$$}

\noindent
To avoid oscillations, the vertical slopes in \eref{2.5}, $(\mq_y)$, are to be computed with the help a nonlinear limiter. In our numerical
experiments reported below, we have used the generalized minmod limiter applied component-wise; see, e.g.,
\cite{LN,NT,Swe}:
\begin{equation*}
{(\mq_y)}_k={\rm minmod}\left(\theta\,\frac{\xbar\mq_{k+1}-\xbar\mq_k}{\dy},\,\frac{\xbar\mq_{k+1}-\xbar\mq_{k-1}}{2\dy},\,
\theta\,\frac{\xbar\mq_k-\xbar\mq_{k-1}}{\dy}\right).
\end{equation*}
The parameter $\theta\in[1,2]$ controls the amount of numerical dissipation: larger values of $\theta$ typically lead to less dissipative
but more oscillatory scheme.  

Equipped with the reconstructed point values 
$\mq_k^{\rm N,S}=\big(\rho_k^{\rm N,S},(\rho v)_k^{\rm N,S},(E+\rho\phi)_k^{\rm N,S}\big)^\top$, we obtain the point values of the other
variables needed in the computation of the numerical fluxes in \eqref{2.2}, that is, 
$$
v_k^{\rm N,S}=\frac{(\rho v)_k^{\rm N,S}}{\rho_k^{\rm N,S}},\quad p_k^{\rm N,S}=(\gamma-1)\left[(E+\rho\phi)_k^{\rm N,S}
-(\rho\phi)_k^{\rm N,S}-\frac{\rho_k^{\rm N,S}(v_k^{\rm N,S})^2}{2}\right],
$$
and
$$
L_k^{\rm N}=p_k^{\rm N}+R_{k+\hf},\quad L_k^{\rm S}=p_k^{\rm S}+R_{k-\hf},
$$
where $\{R_{k+\hf}\}$ are calculated recursively using the midpoint quadrature rule:
 \begin{equation}
R_{k_\ell-\hf}=0,\qquad R_\kph=R_\kmh+\dy\,\xbar\rho_k(\phi_y)_k,\qquad k=k_\ell,\ldots,k_r,
\label{RRR1}
\end{equation}
and $(\rho\phi)_k^{\rm N}=\rho_k^{\rm N}\phi(y_{k+\hf}),\ (\rho\phi)_k^{\rm S}=\rho_k^{\rm S}\phi(y_{k-\hf})$,
and $(\phi_y)_k:=\phi_y(y_k)$.

The terms $b_\kph^\pm$ are one-sided local speeds of propagation, which can be estimated using the smallest and largest eigenvalues of the 
Jacobian $\partial\mg/\partial\mq$:
\begin{equation}
b^+_\kph=\max\left(v_k^{\rm N}+c_k^{\rm N},\,v_{k+1}^{\rm S}+c_{k+1}^{\rm S},\,0\right),\quad
b^-_\kph=\min\left(v_k^{\rm N}-c_k^{\rm N},\,v_{k+1}^{\rm S}-c_{k+1}^{\rm S},\,0\right),
\label{2.9}
\end{equation}
where $c_{k+1}^{\rm N,S}$ are the speeds of sound ($c^2=\gamma p/\rho$). 

Finally, the second term on the right-hand side (RHS) of \eqref{2.2} is a built-in \emph{anti-diffusion term} which involves
$\delta\mq_\kph={\rm minmod}\big(\mq_{k+1}^{\rm S}-\mq_\kph^*,\,\mq_\kph^*-\mq_k^{\rm N}\big)$ using the intermediate state
\begin{equation}
\mq_\kph^*=\frac{b^+_\kph\mq_{k+1}^{\rm S}-b^-_\kph\mq_k^{\rm N}-\big(\mg(\mq_{k+1}^{\rm S})-\mg(\mq_k^{\rm N})\big)}{b^+_\kph-b^-_\kph}.
\label{ad}
\end{equation}

\subsection{Lack of Well-Balancing}\label{sec:lack}
The CU scheme \eref{2.1a}--\eref{ad} is not capable of exactly preserving the steady-state solution \eqref{ss1}. Indeed, substituting 
$v\equiv0$ into \eref{2.1a}--\eref{ad} and noting that for all $k$, $b^+_\kph=-b^-_\kph$ (since $v_k^{\rm N}=v_{k+1}^{\rm S}=0$), we obtain
the ODE system
\begin{equation}
\left\{\begin{aligned}
&\frac{d\,\xbar\rho_k}{dt}=-\frac{1}{\dy}\left[\beta_{k+\hf}(\rho^{\rm S}_{k+1}-\rho^{\rm N}_k-\delta\rho_\kph)-
\beta_{k-\hf}(\rho^{\rm S}_k-\rho^{\rm N}_{k-1}-\delta\rho_\kmh)\right],\\[0.4ex]
&\frac{d(\xbar{\rho v})_k}{dt}=-\frac{1}{2\dy}\left[(L^{\rm S}_{k+1}+L^{\rm N}_k)-(L^{\rm S}_k+L^{\rm N}_{k-1})\right],\\[0.4ex]
&\frac{d(\xbar E_k+\xbar\rho_k\phi_k)}{dt}=-\frac{1}{\dy}\left[\beta_{k+\hf}(E+\rho\phi)^{\rm S}_{k+1}-(E+\rho\phi)^{\rm N}_k
-\delta(E+\rho\phi)_\kph\right.\\
&\hspace*{3.7cm}\left.-\beta_{k-\hf}((E+\rho\phi)^{\rm S}_k-(E+\rho\phi)^{\rm N}_{k-1}-\delta(E+\rho\phi)_\kmh\right],
\end{aligned}\right.
\label{qqq}
\end{equation}
whose RHS does not necessarily vanish and hence the steady state \eqref{ss1} would not be preserved at the discrete level. We would
like to stress that even for the first-order version of the CU scheme \eref{2.1a}--\eref{ad}, that is, when $(\mq_y)_k\equiv0$ in \eref{2.5}, the RHS of \eref{qqq} does not vanish. This means that the lack of balance between the numerical flux and source
terms is a fundamental problem of the scheme. We also note that for smooth solutions, the balance error in \eref{qqq} is expected to be of
order $(\Delta y)^2$, but a coarse grid solution may contain large spurious waves as demonstrated in the numerical experiments presented in
\S\ref{sec4}.

\subsection{Well-Balanced Central-Upwind Scheme}\label{sec22}
In this section, we present a well-balanced modification of the CU scheme. We first introduce a \emph{well-balanced reconstruction}, which
is performed on the equilibrium variables rather than the conservative ones, and then apply a slightly modified CU scheme to the system 
\eref{g11}--\eref{eos1}.

\medskip\noindent
{\bf Well-balanced reconstruction}.
We now describe a special reconstruction, which is used in the derivation of a well-balanced CU scheme. The main idea is to reconstruct
equilibrium variable $L$ rather than $E$. For the first two components we still use  the same piecewise linear reconstructions as before, 
$\widetilde\rho(y)$ and $(\widetilde{\rho v})(y)$, and compute the corresponding point values of 
$\rho_k^{\rm N,S},\ (\rho v)_k^{\rm N,S}$, and $v_k^{\rm N,S}$.

In order to reconstruct $L=p+R$, we use the cell-interface point values $R_{k\pm\hf}$ in \eqref{RRR1} and compute the corresponding 
values of $R$ at the cell centers as
\begin{equation}
R_k=\hf\Big(R_\kmh+R_\kph\Big),\qquad k=k_\ell,\ldots,k_r,
\label{RRR}
\end{equation}
and thus the values of $L$ at the cell centers are
\begin{equation}
L_k=p_k+R_k,
\label{2.15}
\end{equation}
where $p_k=(\gamma-1)\left(\xbar{E}_k-\frac{\xbar\rho_k}{2}v_k^2\right)$ is obtained from the corresponding equation of state (EOS)  and
$v_k=(\xbar{\rho v})_k/\xbar{\rho}_k$. Equipped with \eref{2.15}, we then apply the minmod reconstruction procedure to $\{L_k\}$ and obtain
the point values of $L$ at the cell interfaces:
\begin{equation*}
L_k^{\rm N}=L_k+\frac{\dy}{2}{(L_y)}_k,\quad L_{k+1}^{\rm S}=L_{k+1}-\frac{\dy}{2}{(L_y)}_{k+1},
\end{equation*}
where the slopes $\{(L_y)_k\}$ are computed using the generalized minmod limiter.
Finally, the point values of $p$ and $E$ needed for computation of numerical fluxes are
given by $p_k^{\rm N,S}=L_k^{\rm N,S}-R_{k\pm \frac{1}{2}}$ and 
$E_k^{\rm N,S}=\frac{1}{\gamma-1}\,p_k^{\rm N,S}+\frac{1}{2}\rho_k^{\rm N,S}(v_k^{\rm N,S})^2$, respectively.
\begin{remark}
If the gravitational potential is linear ($\phi(y)=gy$ with $g$ being the gravitational constant), then $R$ can be computed by integrating
the piecewise linear reconstruction of $\rho$, \eref{2.4}, which results in the piecewise quadratic approximation of $R$:
\[
\widetilde R(y)=g\!\!\!\int\limits_{y_{k_\ell-\frac{1}{2}}}^y\!\!\!\widetilde\rho(\xi)\,d\xi=
g\sum_k\Big[\Delta y\sum_{i=k_\ell}^{k-1}\xbar\rho_i+\xbar\rho_k(y-y_{k-\hf})
+\frac{(\rho_y)_k}{2}(y-y_{k-\hf})(y-y_{k+\hf})\Big]\Chi_{C_k}(y).
\]
Then, the point values of $R$ at the cell interfaces and at cell centers are given, respectively, by
\begin{equation*}
R_\kph=g\dy\sum_{i=k_\ell}^k\xbar\rho_i\quad\mbox{and}\quad
R_k=g\Delta y\sum_{i=k_\ell}^{k-1}\xbar\rho_i+\frac{g\dy}{2}\,\xbar\rho_k-\frac{g(\dy)^2}{8}(\rho_y)_k.
\end{equation*}
\end{remark}

\noindent
{\bf Well-balanced evolution}. 
The cell-averages $\xbar\mq_k$ are evolved in time according to the system of ODEs \eref{2.1a}. The second component of the
numerical fluxes $\bm{\mathcal{G}}$ is computed the same way as in \eref{2.2}, but with $\mg$ given by \eref{g11}--\eref{eos1}, that is,
\begin{equation}
\begin{aligned}
\mathcal{G}^{\,(2)}_\kph&=\frac{b^+_\kph\left(\rho_k^{\rm N}(v_k^{\rm N})^2+L_k^{\rm N}\right)-
b^-_\kph\left(\rho_{k+1}^{\rm S}(v_{k+1}^{\rm S})^2+L_{k+1}^{\rm S}\right)}{b^+_\kph-b^-_\kph}\\
&+\beta_\kph\left((\rho v)_{k+1}^{\rm S}-(\rho v)_k^{\rm N}-\delta(\rho v)_\kph\right),
\end{aligned}
\label{f2}
\end{equation}
while the first and third components are modified in order to exactly preserve the steady state:
\begin{equation}
\begin{aligned}
\mathcal{G}^{\,(1)}_\kph&=\frac{b^+_\kph(\rho v)_k^{\rm N}-b^-_\kph(\rho v)_{k+1}^{\rm S}}{b^+_\kph-b^-_\kph}
+\beta_\kph H(\psi_{k+\hf})\cdot\left(\rho_{k+1}^{\rm S}-\rho_k^{\rm N}-\delta\rho_\kph\right),\\[0.8ex]
\mathcal{G}^{\,(3)}_\kph&=\frac{b^+_\kph v_k^{\rm N}(E_k^{\rm N}+(\rho\phi)_k^{\rm N}+p_k^{\rm N})-
b^-_\kph v_{k+1}^{\rm S}(E_{k+1}^{\rm S}+(\rho\phi)_{k+1}^{\rm S}+p_{k+1}^{\rm S})}{b^+_\kph-b^-_\kph}\\
&+\beta_\kph 
\left[E_{k+1}^{\rm S}-E_k^{\rm N}+H(\psi_{k+\hf})\cdot
\left((\rho\phi)_{k+1}^{\rm S}-(\rho\phi)_k^{\rm N}-\delta(E+\rho\phi)_\kph\right)\right],
\end{aligned}
\label{2.20}
\end{equation}
where $\psi_{k+\hf}=\frac{|L_{k+1}-L_k|}{\dy}\cdot\frac{y_{k_r+\frac{1}{2}}-y_{k_\ell-\frac{1}{2}}}{\max\{L_k,L_{k+1}\}}$.

Notice that the last terms on the RHS of \eref{2.20} include a slight modification of the original CU flux, which is now multiplied by a
smooth cut-off function $H$, which is  small when the computed solution is locally (almost) at a steady state, where 
${|L_{k+1}-L_k|}/{\dy}\sim0$, and is otherwise  close to 1. This is done in order to guarantee the well-balanced property of the scheme,
shown in Theorem \ref{thm2.1} below. A sketch of a typical function $H$ is shown in Figure \ref{HHH}. In all of our numerical experiments,
we have used
\begin{equation}
H(\psi)=\frac{(C\psi)^m}{1+(C\psi)^m}, 
\label{hfnc}
\end{equation}
with $C=200$ and $m=6$. To reduce the dependence of the computed solution on the choice of particular values of $C$ and $m$, the
argument of $H$ in \eref{2.20} is normalized by a factor $\frac{y_{k_r+\frac{1}{2}}-y_{k_\ell-\frac{1}{2}}}{\max\{L_k,L_{k+1}\}}$, which
makes $H(\psi_\kph)$ dimensionless. \begin{figure}[ht!]
\centerline{\scalebox{0.54}{\includegraphics{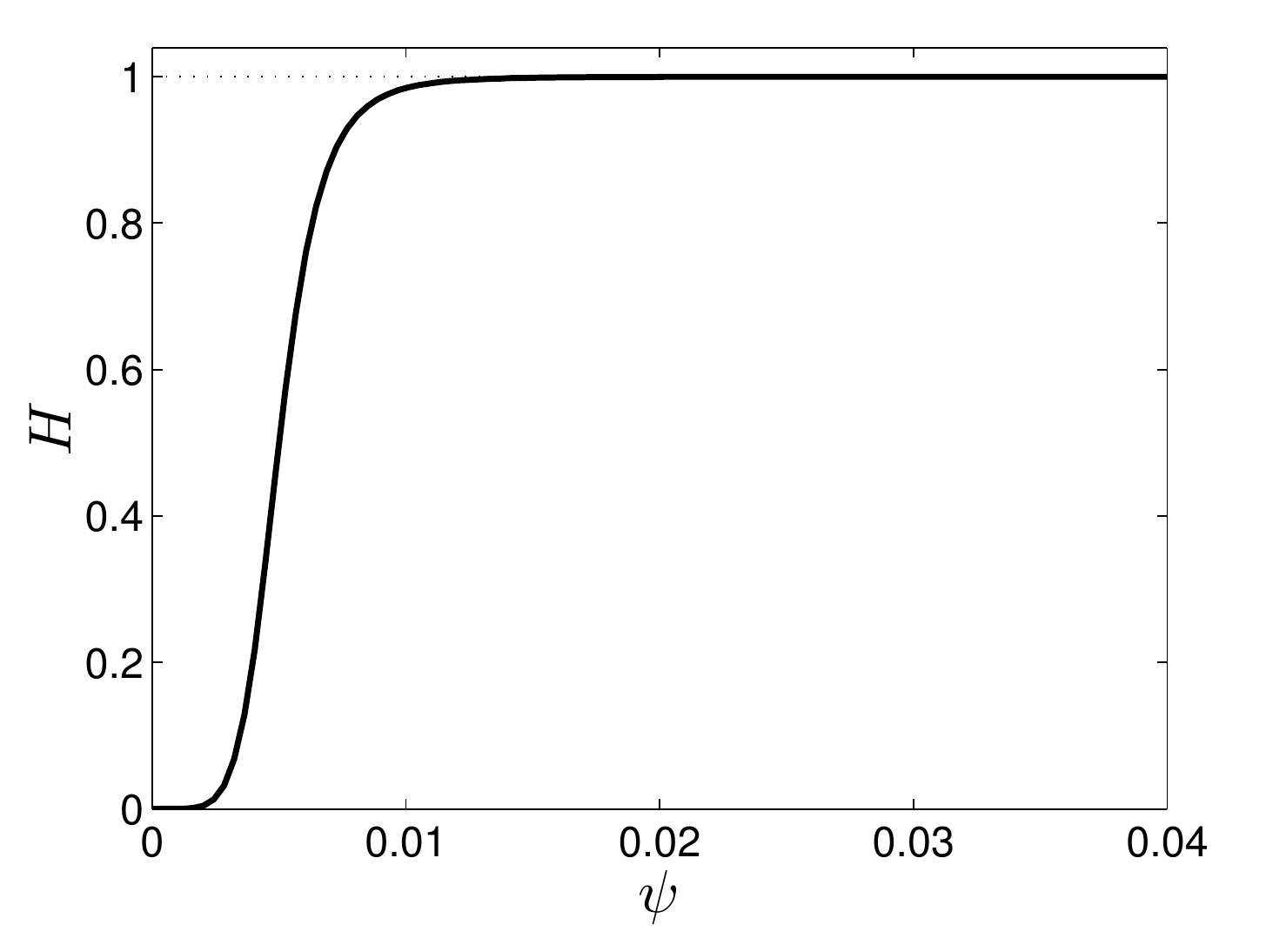}}}
\caption{\sf Sketch of $H(\psi)$.\label{HHH}}
\end{figure}

\begin{theorem}\label{thm2.1}
The semi-discrete CU scheme \eref{2.1a}, \eref{f2}, \eref{2.20} coupled with the reconstruction described in \S\ref{sec22} is well-balanced
in the sense that it exactly preserves the steady state \eqref{ss1}.
\end{theorem}
\begin{proof}
Assume that at certain time level, we have
\begin{equation}
v_k^{\rm N}\equiv v_k\equiv v_k^{\rm S}\equiv0\quad\mbox{and}\quad L_k^{\rm N}\equiv L_k\equiv L_k^{\rm S}\equiv\widehat L,
\label{data}
\end{equation}
where $\widehat L$ is a constant. In order to prove that the proposed scheme is well-balanced, we will show that the numerical fluxes are
constant for all $k$ for the data in \eref{data} and thus need the RHS of \eref{2.1a} is identically equal to zero at such steady states.

Indeed, the first component in \eref{2.20} of the numerical flux vanish since $v_k^{\rm N}=v_{k+1}^{\rm S}=0$ and $L_k=L_{k+1}=\widehat L$
(the latter implies $H(\psi_\kph)=H(0)=0$). The second component \eref{f2} of the numerical flux is constant and equal to $\widehat L$
since $v_k^{\rm N}=v_{k+1}^{\rm S}=0$ and $L_k^{\rm N}=L_{k+1}^{\rm S}=\widehat L$. Finally, the third component in \eqref{2.20} of the
numerical flux also vanishes:
\begin{equation*}
\begin{aligned}
\mathcal{G}^{(3)}_\kph&=\beta_\kph
\left[E_{k+1}^{\rm S}-E_k^{\rm N}+H(\psi_{k+\hf})\cdot
\left((\rho\phi)_{k+1}^{\rm S}-(\rho\phi)_k^{\rm N}-\delta(E+\rho\phi)_\kph\right)\right]\\
&=\frac{\beta_\kph}{\gamma-1}\cdot\frac{p_{k+1}^{\rm S}-p_k^{\rm N}}{2}
=\frac{\beta_\kph}{2(\gamma-1)}\left[\big(L_{k+1}^{\rm S}-R_\kph)-(L_k^{\rm N}-R_\kph\big)\right]=0,
\end{aligned}
\end{equation*}
since $L_k^{\rm N}=L_{k+1}^{\rm S}=\widehat L$.$\hfill\blacksquare$
\end{proof}

\section{Two-Dimensional Well-Balanced Central-Upwind Scheme}\label{sec3} 
In this section, we describe the well-balanced semi-discrete CU scheme for the 2-D Euler equations with gravitation \eqref{eq:global}, which
can be rewritten in the following vector form:
\begin{equation*}
\mq_t+\mf(\mq)_x+\mg(\mq)_y=\mo,
\end{equation*}
where $\mq:=(\rho,\rho u,\rho v,E+\rho\phi)^\top$, the fluxes are 
\begin{equation*}
\begin{aligned}
&\mf(\mq):=(\rho u,\rho u^2+K,\rho uv,u(E+\rho\phi+p))^\top,\quad
\mg(\mq):=(\rho v,\rho uv,\rho v^2+L,v(E+\rho\phi+p))^\top,
\end{aligned}
\end{equation*}
in which $K=p+Q$ and $L=p+R$ with the global variables $Q$ and $R$ introduced in \eref{QR}.

Let $C_{j,k}:=[x_\jmh,x_\jph]\times[y_\kmh,y_\kph]$ denote the 2-D cells and we assume that the cell averages of the computed numerical
solution,
\begin{equation*}
\xbar\mq_{j,k}(t):=\frac{1}{\dx\dy}\iint\limits_{C_{j,k}}\mq(x,y,t)\,dx\,dy,\quad  ~j=j_\ell,\ldots,j_r,~k=k_\ell,\ldots,k_r
\end{equation*}
are available at time level $t$ and describe their CU evolution.

\bigskip
\noindent
{\bf Well-balanced reconstruction}.
Similarly to the 1-D case, we reconstruct only the first three components of the conservative variables, $\rho,\rho u,\rho v$,
\begin{equation}
\widetilde q^{\,(i)}(x,y)=\xbar q^{\,(i)}_{j,k}+(q^{(i)}_x)_{j,k}(x-x_j)+(q^{(i)}_y)_{j,k}(y-y_k),\quad(x,y)\in C_{j,k},~~i=1,2,3,
\label{2.4}
\end{equation}
and obtain the corresponding point values at the four cell interfaces $(x_\jpmh,y_k)$ and $(x_j,y_\kpmh)$ denoted 
$(q^{(i)})^{\rm E,W,N,S}_{j,k}$ with the slopes $(q^{(i)}_x)_{j,k}$ and $(q^{(i)}_y)_{j,k}$ being computed using the generalized minmod 
limiter.

The point values of the energy $p$ and $E$ should be calculated from the new equilibrium variables obtained from the reconstruction of $K$ 
and $L$. We emphasize  that since at the steady states \eref{1.11}, $K=K(y)$ is independent of $x$ and $L=L(x)$ is independent of $y$, we,
in fact, perform 1-D reconstructions for $K$ and $L$ in the $x$- and $y$-directions, respectively. To this end, we first use the midpoint 
rule to compute the point values of the integrals $Q$ and $R$ in \eref{QR} at the cell interfaces in the $x$- and $y$-directions, respectively:
\begin{equation}\label{RRR2}
\begin{aligned}
Q_{j_\ell-\hf,k}&=0,\quad~
\left\{\begin{aligned}
&Q_{\jph,k}=Q_{\jmh,k}+\dx\,\xbar\rho_{j,k}(\phi_x)_{j,k},\\
&Q_{j,k}=\hf\left\{Q_{j-\hf,k}+Q_{j+\hf,k}\right\},
\end{aligned}\right.
\quad j=j_\ell,\ldots,j_r,~k=k_\ell,\ldots,k_r,\\[0.4ex]
R_{j,k_\ell-\frac{1}{2}}&=0,\quad~
\left\{\begin{aligned}
&R_{j,\kph}=R_{j,\kmh}+\dy\,\xbar\rho_{j,k}(\phi_y)_{j,k},\\
&R_{j,k}=\hf\left\{R_{j,k-\hf}+R_{j,k+\hf}\right\},
\end{aligned}\right.
\quad j=j_\ell,\ldots,j_r,~k=k_\ell,\ldots,k_r,
\end{aligned}
\end{equation}
where $(\phi_x)_{j,k}:=\phi_x(x_j,y_k)$ and $(\phi_y)_{j,k}:=\phi_y(x_j,y_k)$. We then compute the cell center values 
$K_{j,k}= p_{j,k}+Q_{j,k}$ and $L_{j,k}=p_{j,k}+R_{j,k}$, and reconstruct the point values of $K$ and $L$ at the cell interfaces,
\begin{equation}
\begin{aligned}
K_{j,k}^{\rm E}&=K_{j,k}+\frac{\dx}{2}(K_x)_{j,k},&K_{j,k}^{\rm W}&=K_{j,k}-\frac{\dx}{2}(K_x)_{j,k},\\
L_{j,k}^{\rm N}&=L_{j,k}+\frac{\dy}{2}(L_y)_{j,k},&L_{j,k}^{\rm S}&=L_{j,k}-\frac{\dy}{2}(L_y)_{j,k},
\end{aligned}
\label{3.2}
\end{equation}
where $(K_x)_{j,k}$ and $(L_y)_{j,k}$ are the minmod limited slopes in the $x$- and $y$- directions, respectively.
Finally, the values obtained in \eref{3.2} are used to evaluate the point values of the pressure
\begin{equation*}
\begin{aligned}
p_{j,k}^{\rm E}&=K_{j,k}^{\rm E}-Q_{\jph,k},&p_{j,k}^{\rm W}&=K_{j,k}^{\rm W}-Q_{\jmh,k},\\
p_{j,k}^{\rm N}&=L_{j,k}^{\rm N}-R_{j,\kph},&p_{j,k}^{\rm S}&=L_{j,k}^{\rm S}-R_{j,\kmh},
\end{aligned}
\end{equation*}
and then the corresponding point values of $E$ are calculated from the EOS, $E=\frac{p}{\gamma-1}+\hf\rho(u^2+v^2)$.

We then estimate the one-sided local speeds of propagation in the $x$- and $y$- directions, respectively, using the smallest and largest
eigenvalues of the Jacobians $\partial\mf/\partial\mq$ and $\partial\mg/\partial\mq$:
\begin{equation*}
\begin{aligned}
a^+_{\jph,k}&=\max\left(u^{\rm E}_{j,k}+c^{\rm E}_{j,k},\,u^{\rm W}_{j+1,k}+c^{\rm W}_{j+1,k},\,0\right),\quad
&a^-_{\jph,k}&=\min\left(u^{\rm E}_{j,k}-c^{\rm E}_{j,k},\,u^{\rm W}_{j+1,k}-c^{\rm W}_{j+1,k},\,0\right),\\[0.3ex]
b^+_{j,\kph}&=\max\left(v^{\rm N}_{j,k}+c^{\rm N}_{j,k},\,v^{\rm S}_{j,k+1}+c^{\rm S}_{j,k+1},\,0\right),\quad
&b^-_{j,\kph}&=\min\left(v^{\rm N}_{j,k}-c^{\rm N}_{j,k},\,v^{\rm S}_{j,k+1}-c^{\rm S}_{j,k+1},\,0\right),
\end{aligned}
\end{equation*}
in terms of  the velocities, $u^{\rm E,W,N,S}_{j,k}$, and the speeds of sound, $c^{\rm E,W,N,S}_{j,k}$.

\medskip\noindent
{\bf Well-balanced evolution}. 
The cell-averages $\,\xbar\mq$ are evolved in time according to the following system of ODEs:
\begin{equation}
\frac{d}{dt}\,\xbar\mq_{j,k}=-\frac{\bm{\mathcal{F}}_{\jph,k}-\bm{\mathcal{F}}_{\jmh,k}}{\dx}-
\frac{\bm{\mathcal{G}}_{j,\kph}-\bm{\mathcal{G}}_{j,\kmh}}{\dy},
\label{3.10}
\end{equation}
where $\bm{\mathcal{F}}$ and $\bm{\mathcal{G}}$ are numerical fluxes. Introducing the notations
$$
\alpha_{\jph,k}:=\frac{a^+_{\jph,k}a^-_{\jph,k}}{a^+_{\jph,k}-a^-_{\jph,k}}\quad\mbox{and}\quad
\beta_{j,\kph}:=\frac{b^+_{j,\kph}b^-_{j,\kph}}{b^+_{j,\kph}-b^-_{j,\kph}},
$$
we write the components of $\bm{\mathcal {F}}_{\jph,k}$ and $\bm{\mathcal {G}}_{j,\kph}$ as
\begin{align*}
\mathcal{F}^{(1)}_{\jph,k}=&\,\frac{a^+_{\jph,k}(\rho u)_{j,k}^{\rm E}-a^-_{\jph,k}(\rho u)_{j+1,k}^{\rm W}}
{a^+_{\jph,k}-a^-_{\jph,k}}+\alpha_{\jph,k}H(\psi_{\jph,k})\cdot
\big(\rho_{j+1,k}^{\rm W}-\rho_{j,k}^{\rm E}-\delta\rho_{\jph,k}\big),\\[0.6ex]
\mathcal{F}^{(2)}_{\jph,k}=&\,\frac{a^+_{\jph,k}\left(\rho_{j,k}^{\rm E}(u_{j,k}^{\rm E})^2+K_{j,k}^{\rm E}\right)-
a^-_{\jph,k}(\rho_{j+1,k}^{\rm W}(u_{j+1,k}^{\rm W})^2+K_{j+1,k}^{\rm W})}{a^+_{\jph,k}-a^-_{\jph,k}}\\[0.5ex]
&+\,\alpha_{\jph,k}\big((\rho u)_{j+1,k}^{\rm W}-(\rho u)_{j,k}^{\rm E}-\delta(\rho u)_{\jph,k}\big),\\[0.8ex]
\mathcal{F}^{(3)}_{\jph,k}=&\,\frac{a^+_{\jph,k}\rho_{j,k}^{\rm E}u_{j,k}^{\rm E}v_{j,k}^{\rm E}-
a^-_{\jph,k}\rho_{j+1,k}^{\rm W}u_{j+1,k}^{\rm W}v_{j+1,k}^{\rm W}}{a^+_{\jph,k}-a^-_{\jph,k}}+
\alpha_{\jph,k}\big((\rho v)_{j+1,k}^{\rm W}-(\rho v)_{j,k}^{\rm E}-\delta(\rho v)_{\jph,k}\big),\\[0.8ex]
\mathcal{F}^{(4)}_{\jph,k}=&\frac{a^+_{\jph,k}u_{j,k}^{\rm E}(E_{j,k}^{\rm E}+(\rho\phi)_{j,k}^{\rm E}+p_{j,k}^{\rm E})-
a^-_{\jph,k}u_{j+1,k}^{\rm W}(E_{j+1,k}^{\rm W}+(\rho\phi)_{j+1,k}^{\rm W}+p_{j+1,k}^{\rm W})}{a^+_{\jph,k}-a^-_{\jph,k}}\\[0.5ex]
&+\alpha_{\jph,k}\Big[E_{j+1,k}^{\rm W}-E_{j,k}^{\rm E}+H(\psi_{\jph,k})\cdot\big((\rho\phi)_{j+1,k}^{\rm W}
-(\rho\phi)_{j,k}^{\rm E}-\delta(E+\rho\phi)_{\jph,k}\big)\Big],\\[0.8ex]
\mathcal{G}^{(1)}_{j,\kph}=&\,\frac{b^+_{j,\kph}(\rho v)_{j,k}^{\rm N}-b^-_{j,\kph}(\rho v)_{j,k+1}^{\rm S}}
{b^+_{j,\kph}-b^-_{j,\kph}}+\beta_{j,\kph}H(\psi_{j,\kph})\cdot
\big(\rho_{j,k+1}^{\rm S}-\rho_{j,k}^{\rm N}-\delta\rho_{j,\kph}\big),\\[0.8ex]
\mathcal{G}^{(2)}_{j,\kph}=&\,\frac{b^+_{j,\kph}\rho_{j,k}^{\rm N}u_{j,k}^{\rm N}v_{j,k}^{\rm N}-
b^-_{j,\kph}\rho_{j,k+1}^{\rm S}u_{j,k+1}^{\rm S}v_{j,k+1}^{\rm S}}{b^+_{j,\kph}-b^-_{j,\kph}}+
\beta_{j,\kph}\big((\rho u)_{j,k+1}^{\rm S}-(\rho u)_{j,k}^{\rm N}-\delta(\rho u)_{j,\kph}\big),\\[0.8ex]
\mathcal{G}^{(3)}_{j,\kph}=&\,\frac{b^+_{j,\kph}\left(\rho_{j,k}^{\rm N}(v_{j,k}^{\rm N})^2+L_{j,k}^{\rm N}\right)-
b^-_{j,\kph}\left(\rho_{j,k+1}^{\rm S}(v_{j,k+1}^{\rm S})^2+L_{j,k+1}^{\rm S}\right)}{b^+_{j,\kph}-b^-_{j,\kph}}\\[0.5ex]
& +\,\beta_{j,\kph}\big((\rho v)_{j,k+1}^{\rm S}-(\rho v)_{j,k}^{\rm N}-\delta(\rho v)_{j,\kph}\big),\\[0.8ex]
\mathcal{G}^{(4)}_{j,\kph}=&\frac{b^+_{j,\kph}v_{j,k}^{\rm N}(E_{j,k}^{\rm N}+(\rho\phi)_{j,k}^{\rm N}+p_{j,k}^{\rm N})-
b^-_{j,\kph}v_{j,k+1}^{\rm S}(E_{j,k+1}^{\rm S}+(\rho\phi)_{j,k+1}^{\rm S}+p_{j,k+1}^{\rm S})}{b^+_{j,\kph}-b^-_{j,\kph}}\\[0.5ex]
&+\beta_{j,\kph}\big(E_{j,k+1}^{\rm S}-E_{j,k}^{\rm N}+H(\psi_{j,k+\hf})\cdot\big((\rho\phi)_{j,k+1}^{\rm S}
-(\rho\phi)_{j,k}^{\rm N}-\delta(E+\rho\phi)_{j,\kph}\big),
\end{align*}
where
$(\rho\phi)_{j,k}^{\rm E}=\rho_{j,k}^{\rm E}\phi(x_{j+\hf},y_k),\ (\rho\phi)_{j,k}^{\rm W}=\rho_{j,k}^{\rm W}\phi(x_{j-\hf},y_k)$,
$(\rho\phi)_{j,k}^{\rm N}=\rho_{j,k}^{\rm N}\phi(x_j,y_{k+\hf}),\ (\rho\phi)_{j,k}^{\rm S}=\rho_{j,k}^{\rm S}\phi(x_j,y_{k-\hf})$,
$\psi_{j,k+\hf}=\frac{|K_{j+1,k}-K_{j,k}|}{\dx}\cdot\frac{x_{k_r+\frac{1}{2}}-x_{k_\ell-\frac{1}{2}}}{\max_{j,k}\{K_{j,k},K_{j+1,k}\}}$,
$\psi_{j+\hf,k}=\frac{|L_{j,k+1}-L_{j,k}|}{\dy}\cdot\frac{y_{k_r+\frac{1}{2}}-y_{k_\ell-\frac{1}{2}}}{\max_{j,k}\{L_{j,k},L_{j,k+1}\}}$, 
the function $H$ is defined, as before, in \eref{hfnc}, and
\begin{equation}
\delta\mq_{\jph,k}={\rm minmod}\big(\mq_{j+1,k}^{\rm W}-\mq_{\jph,k}^*,\,\mq_{\jph,k}^*-\mq_{j,k}^{\rm E}\big)
\label{adx}
\end{equation}
and
\begin{equation}
\delta\mq_{j,\kph}={\rm minmod}\big(\mq_{j,k+1}^{\rm S}-\mq_{j,\kph}^*,\,\mq_{j,\kph}^*-\mq_{j,k}^{\rm N}\big)
\label{ady}
\end{equation}
are build-in anti-diffusion terms with
$$
\mq_{\jph,k}^*=\frac{a^+_{\jph,k}\mq_{j+1,k}^{\rm W}-a^-_{\jph,k}\mq_{j,k}^{\rm E}-\big\{\mf(\mq_{j+1,k}^{\rm W})-\mf(\mq_{j,k}^{\rm E})
\big\}}{a^+_{\jph,k}-a^-_{\jph,k}}
$$
and
$$
\mq_{j,\kph}^*=\frac{b^+_{j,\kph}\mq_{j,k+1}^{\rm S}-b^-_{j,\kph}\mq_{j,k}^{\rm N}-\big\{\mg(\mq_{j,k+1}^{\rm S})-\mg(\mq_{j,k}^{\rm N})
\big\}}{b^+_{j,\kph}-b^-_{j,\kph}}.
$$
Notice that the anti-diffusion terms \eref{adx} and \eref{ady} can be rigorously derived from the fully discrete CU framework along the 
lines of (though slightly different from) \cite{KLin}.
{\color{black} Namely, 
while $\delta\mq_{\jph,k}$ and $\delta\mq_{j,\kph}$ in \cite{KLin} are obtained by computing the slopes in the piecewise linear
interpolant over the ``side'' domains $D_{\jph,k}$ and $D_{j,\kph}$ 
using the point values at their corners,
here the corresponding  terms are obtained from the same interpolants using the values at the midpoints of the long sides of the rectangles
$D_{\jph,k}$ and $D_{j,\kph}$. 
}

We can now state the following well-balanced property of the proposed 2-D CU scheme, whose proof is similar to that of Theorem \ref{thm2.1}
{\color{black}and thus it is presented in Appendix \ref{prf}}.
\begin{theorem}\label{thm3.1}
The 2-D semi-discrete CU scheme described  above is well-balanced in the sense that it exactly
preserves the steady state \eref{1.11}.
\end{theorem}
\ifx
\begin{proof}
The proof is similar to the proof of Theorem \ref{thm2.1}.$\hfill\blacksquare$
\end{proof}
\fi

\section{Numerical Examples}\label{sec4}
In this section, we present a number of 1-D and 2-D numerical examples, in which we demonstrate the performance of the proposed 
well-balanced semi-discrete CU scheme. 

In all of the examples below, we have used the three-stage third-order strong stability preserving (SSP) Runge-Kutta method (see, e.g.,
\cite{GKS,GST,SO}) to solve the ODE systems \eref{2.1a} and \eref{3.10}. The CFL number has been set to 0.4. Also, we have used the
following constant values: the minmod parameter $\theta=1.3$ and the specific heat ratio $\gamma=1.4$. 

In Examples 2--4, the initial data correspond to the 1-D and 2-D steady-state solutions and their small perturbations. The designed
well-balanced CU schemes are capable of exactly preserving discrete versions of the 1-D and 2-D steady states \eqref{ss1} and \eqref{1.11},
respectively. In Appendix \ref{app1}, we provide a detailed description of how to construct such steady states.

\subsection{One-Dimensional Examples}\label{sec41}
In all of the 1-D numerical experiments, we use a uniform mesh with the total number of grid cells $N=k_r-k_\ell+1$.

\paragraph{Example 1---Shock Tube Problem.}
The first example is a modification of the Sod shock tube problem taken from \cite{LXL,XinShu13}. We solve the system
\eref{g11}--\eref{eos1} with $\phi(y)=y$ in the computational domain $[0,1]$ using the following initial data:
\begin{equation*}
(\rho(y,0),v(y,0),p(y,0))=\left\{\begin{array}{lr}(1,0,1),&\textrm{ if }y\le0.5,\\(0.125,0,0.1),&\textrm{ if }y>0.5,\end{array}\right.
\end{equation*}
and reflecting boundary conditions at both ends of the computational domain. These boundary conditions are implemented using the ghost
cell technique as follows:
$$
\begin{aligned}
&\xbar\rho_{k_\ell-1}:=\xbar\rho_{k_\ell},\quad v_{k_\ell-1}:=-v_{k_\ell},\quad L_{k_\ell-1}:=L_{k_\ell},\\
&\xbar\rho_{k_r+1}:=\xbar\rho_{k_r},\quad v_{k_r+1}:=-v_{k_r},\quad L_{k_r+1}:=L_{k_r}.
\end{aligned}
$$
We compute the solution using $N=100$ uniformly placed grid cells and compare it with the reference solution obtained using $N=2000$ uniform
cells. In Figure \ref{fig41}, we plot both the coarse and fine grid solutions at time $T=0.2$. As one can see, the proposed CU scheme
captures the solutions on coarse mesh quite well showing a good agreement with both the reference solution and the results obtained in 
\cite{LXL,XinShu13}.
\begin{figure}[ht!]
\centerline{\scalebox{0.40}{\includegraphics{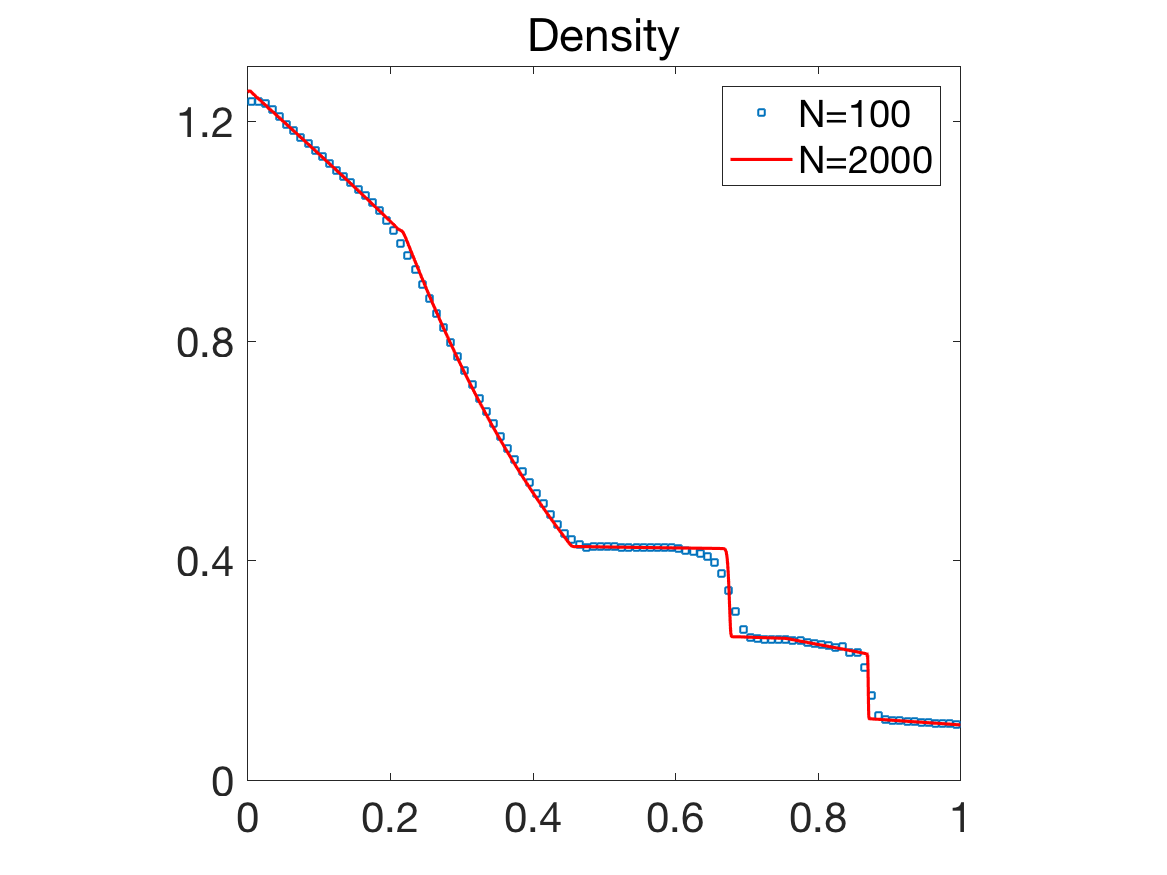}\hspace*{-0.0cm}{\includegraphics{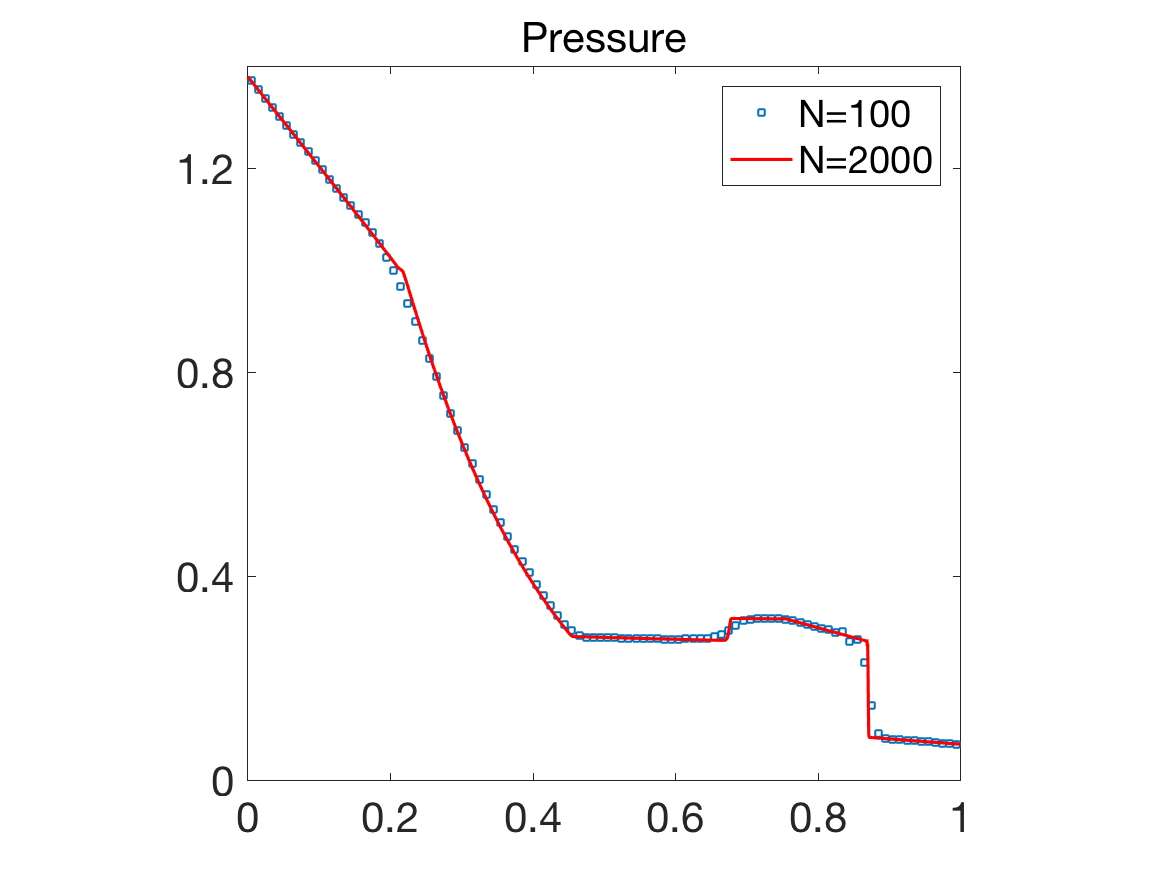}}}}
\vspace*{0.4cm}
\centerline{\scalebox{0.40}{\includegraphics{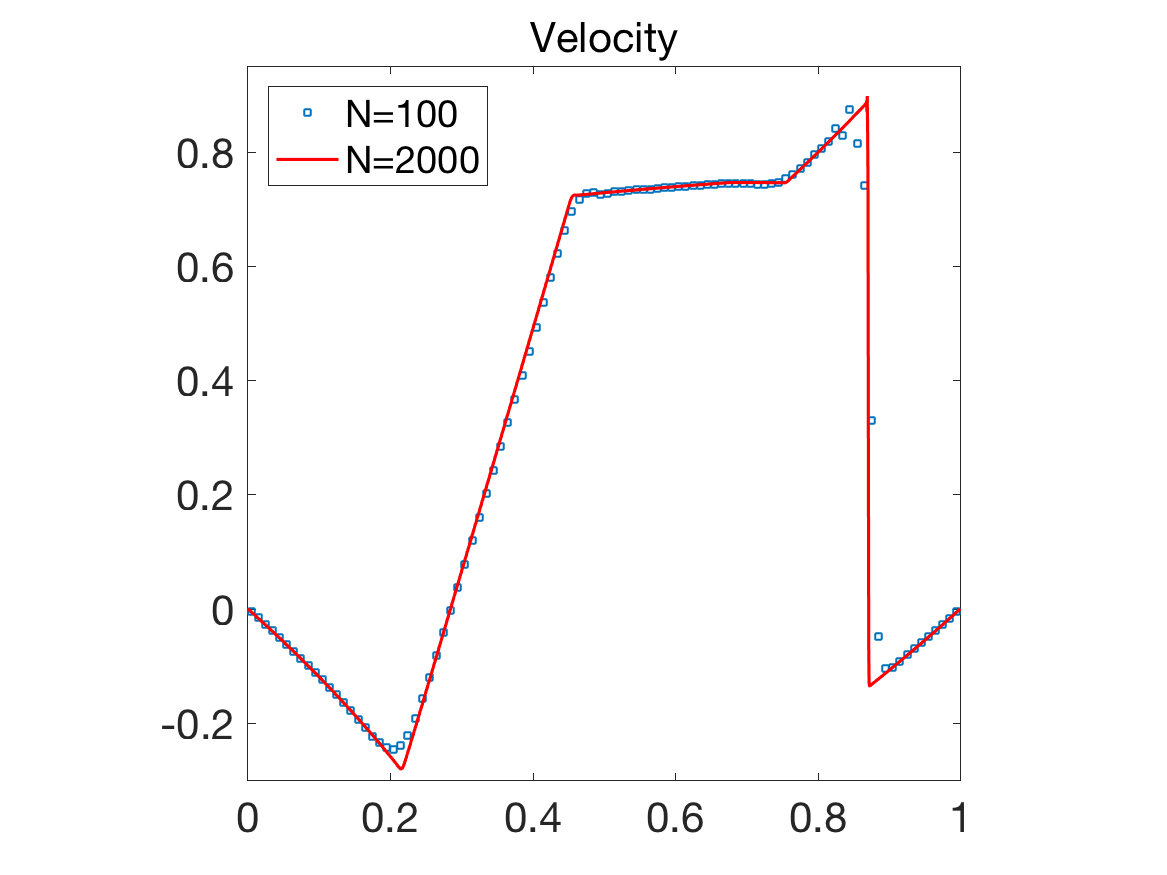}\hspace*{-0.0cm}{\includegraphics{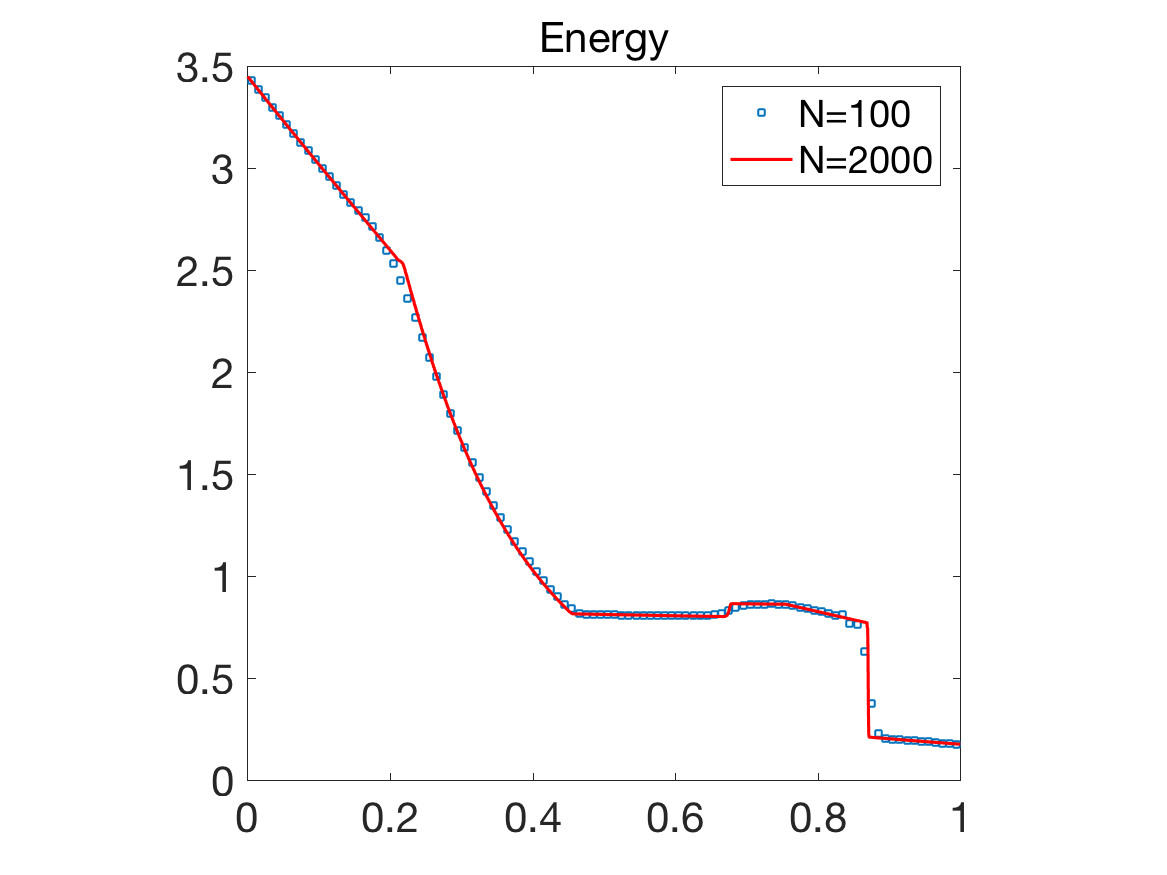}}}}
\caption{\sf Example 1: Solutions computed by the well-balanced CU scheme using $N=100$ and 2000 cells.\label{fig41}}
\end{figure}

{\color{black}Notice that the coarse mesh solution features some oscillations in the vicinity of the shock. These oscillations, however, seem
to be the so-called ``WENO-type oscillations'' as they disappear when the mesh is refined.}

\paragraph{Example 2---Isothermal Equilibrium Solution.} In the second example, taken from \cite{XinShu13} (see also
\cite{LeVBale,LXL,TXCD}), we test the ability of the proposed CU scheme to accurately capture small perturbations of the steady state
\begin{equation}
\rho(y)=e^{-\phi(y)},\quad v(y)\equiv0,\quad p(y)=e^{-\phi(y)},
\label{icss}
\end{equation}
of the system \eref{g11}--\eref{eos1} with the linear gravitational potential $\phi(y)=y$ (in fact, we use a discrete version of this 
steady state with $L(y)\equiv1$, obtained as described in Appendix \ref{a1}).

We take the computational domain $[0,1]$ and use a zero-order extrapolation at the boundaries:
$$
\begin{aligned}
\begin{aligned}
&\xbar\rho_{k_\ell-1}:=\xbar\rho_{k_\ell}e^{\dy(\phi_y)_{k_\ell}},\quad v_{k_\ell-1}:=v_{k_\ell},\quad L_{k_\ell-1}:=L_{k_\ell},\\
&\xbar\rho_{k_r+1}:=\xbar\rho_{k_r}e^{-\dy(\phi_y)_{k_r}},\quad v_{k_r+1}:=v_{k_r},\quad L_{k_r+1}:=L_{k_r}.
\end{aligned}
\end{aligned}
$$
Note that the boundary conditions on $L$ can be recast in terms of $p$ and $\rho$ as
$$
p_{k_\ell-1}=p_{k_\ell}+\dy\,\xbar\rho_{k_\ell}(\phi_y)_{k_\ell},\quad p_{k_r+1}=p_{k_r}-\dy\,\xbar\rho_{k_r}(\phi_y)_{k_r}.
$$

We first numerically verify that the proposed CU scheme is capable of exactly preserving the steady state \eref{icss}. We use several
uniform grids {\color{black} at time $T=1$} and observe that the initial conditions are preserved within the machine accuracy, while the errors in the non-well-balanced 
computations are of the second order of accuracy as shown in Table \ref{ex2-table}.
\begin{table}[ht!]
\centerline{
\begin{tabular}{|c|c|c|c|c|c|c|}\hline 
$N$&${\|\rho(\cdot,1)-\rho(\cdot,0)\|}_1$&rate&${\|(\rho v)(\cdot,1)-(\rho v)(\cdot,0)\|}_1$&rate&${\|E(\cdot,1)-E(\cdot,0)\|}_1$&rate\\ 
\hline 
100&1.47E-06&--&3.62E-06&--& 4.25E-06&--\\
200&3.79E-07&1.95&9.07E-07&1.99&1.08E-06&1.99\\
400&9.62E-08&1.98&2.27E-07&1.99&2.72E-07&1.97\\
800&2.42E-08&1.99&5.69E-08&1.99&6.84E-08&1.99\\
\hline 
\end{tabular} 
}
\caption{\sf Example 2: $L^1$-errors and corresponding experimental convergence rates in the non-well-balanced computation of $\rho$,
$\rho v$ and $E$; $\phi(y)=y$.\label{ex2-table}}
\end{table}

Next, we introduce a small initial pressure perturbation and consider the system \eref{g11}--\eref{eos1} subject to the following initial
data:
\begin{equation*}
\rho(y,0)=e^{-y},\quad v(y,0)\equiv0,\quad p(y,0)=e^{-y}+\eta e^{-100(y-0.5)^2},
\end{equation*}
where $\eta$ is a small positive number. In the numerical experiments, we use larger ($\eta=10^{-2}$) and smaller ($\eta=10^{-6}$)
perturbations.

We first apply the proposed well-balanced CU scheme to this problem and compute the solution at time $T=0.25$. The obtained pressure
perturbation ($p(y,0.25)-e^{-y}$) computed using $N=200$ and 2000 (reference solution) uniform grid cells are plotted in Figure \ref{fig42}
for both $\eta=10^{-2}$ and $10^{-6}$. As one can see, the scheme accurately captures both small and large perturbations on a relatively
coarse mesh with $N=200$. In order to demonstrate the importance of the well-balanced property, we apply the non-well-balanced CU scheme
described in the beginning of \S\ref{sec2} to the same initial-boundary value problem (IBVP). The obtained results are shown in Figure
\ref{fig42} as well. It should be observed that while the larger perturbation is quite accurately computed by both schemes, the 
non-well-balanced CU scheme fails to accurately capture the smaller one.
\begin{figure}[ht!]
\centerline{\scalebox{0.45}{\includegraphics[width=1.1\textwidth]{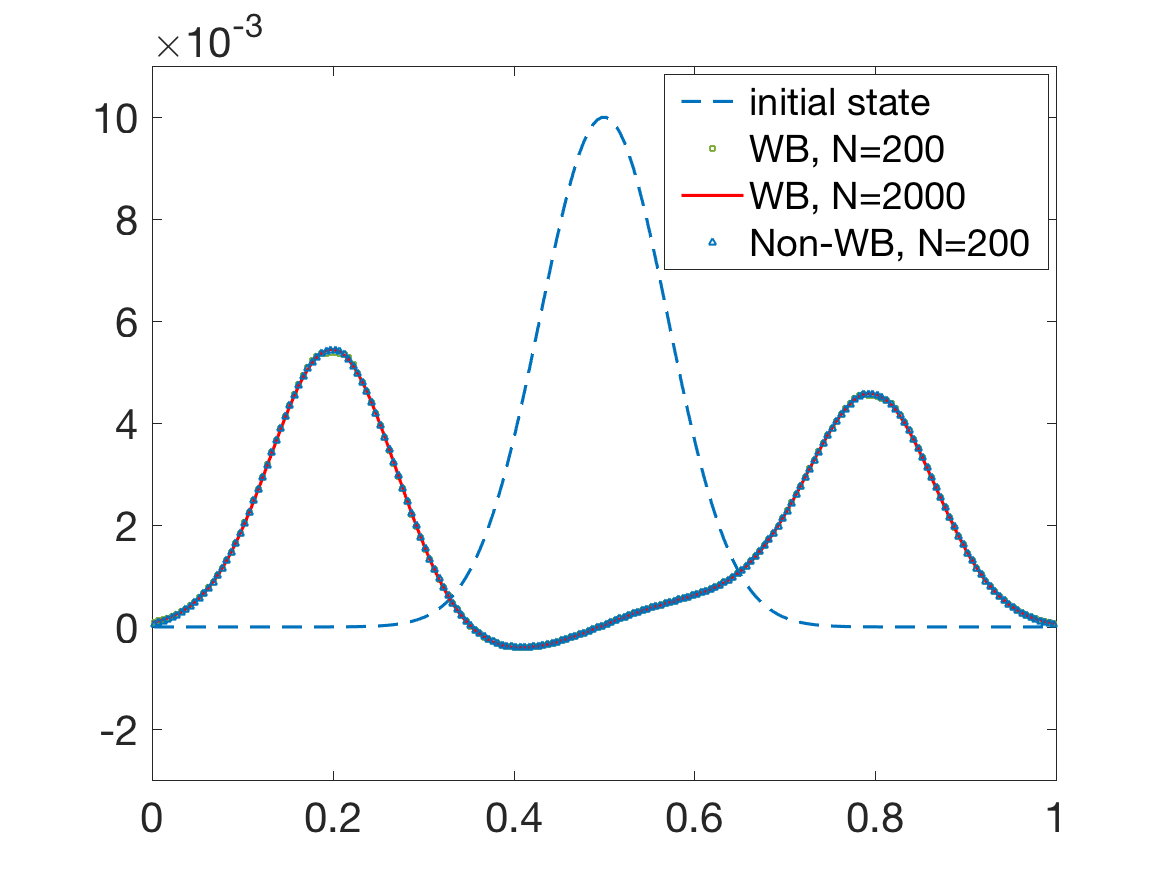}\hspace*{-0.0cm}{\includegraphics[width=1.1\textwidth]{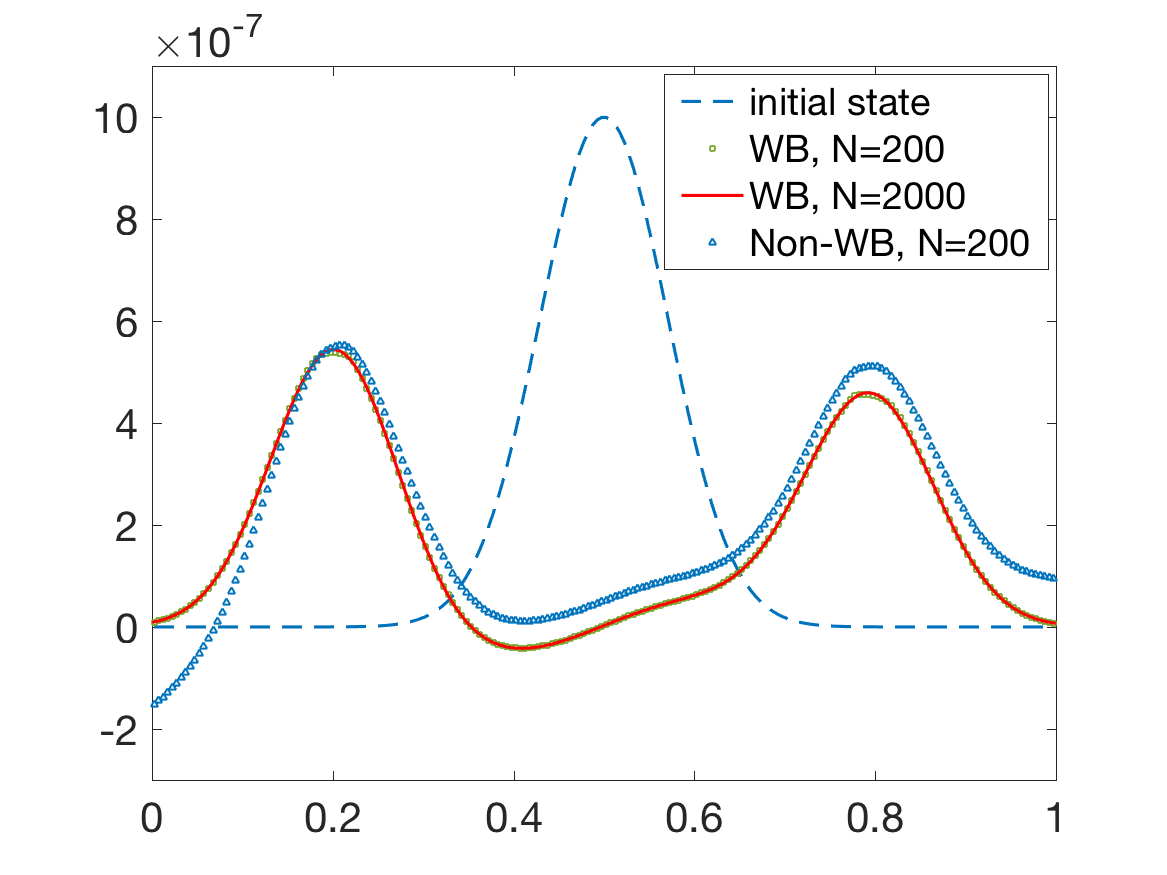}}}}
\caption{\sf Example 2: Pressure perturbation ($p(y,0.25)-e^{-y}$) computed by the well-balanced (WB) and non-well-balanced (Non-WB) CU
schemes with $N=200$ and 2000 for $\eta=10^{-2}$ (left) and $\eta=10^{-6}$ (right).\label{fig42}}
\end{figure}

{\color{black} Finally, we test the experimental rate of convergence of the proposed well-balanced CU scheme by computing
$$
\mbox{rate}_N:=\log_2\left(\frac{||q^{(i)}_N-q^{(i)}_{2N}||_1}{||q^{(i)}_{2N}-q^{(i)}_{4N}||_1}\right),\quad i=1,2,3,
$$
where $q^{(i)}_N$ denotes the $i$th component of the solution computed using a uniform mesh with $N$ cells. The obtained results, presented
in Tables \ref{ex2-tableL1} and \ref{ex2-tableL106}, demonstrate that the experimental rate of convergence is close to the expected
second-order one for both $\eta=10^{-2}$ and $\eta=10^{-6}$. 
\begin{table}[ht!]
\centerline{
\begin{tabular}{|c|c|c|c|c|c|c|}\hline 
$N$&${\|\rho_N-\rho_{2N}\|}_1$&rate$_N$&${\|(\rho v)_N-(\rho v)_{2N}\|}_1$&rate$_N$&${\|E_N-E_{2N}\|}_1$&rate$_N$\\ 
\hline 
100& 2.04E-05  &  -- &2.13E-05  &  -- & 7.67E-05  & --\\
200& 5.45E-06  &  1.90 &5.20E-06  &  2.03&1.91E-05  &  2.00\\
400& 1.33E-06  &  2.03 &1.08E-06  &  2.25&4.13E-06  &  2.21\\
800& 3.53E-07  &  1.91 &2.36E-07  &  2.20&9.46E-07  &  2.12\\
\hline 
\end{tabular} 
}
\caption{\sf Example 2: $L^1$-errors and experimental convergence rates for the well-balanced CU scheme; $\eta=10^{-2}$.\label{ex2-tableL1}}
\end{table}
\begin{table}[ht!]
\centerline{
\begin{tabular}{|c|c|c|c|c|c|c|}\hline 
$N$&${\|\rho_N-\rho_{2N}\|}_1$&rate$_N$&${\|(\rho v)_N-(\rho v)_{2N}\|}_1$&rate$_N$&${\|E_N-E_{2N}\|}_1$&rate$_N$\\ 
\hline 
100 & 3.53E-09  &  -- &3.52E-09 &  -- & 2.43E-09 & --\\
200 & 1.01E-09  &  1.80 &8.53E-10  &  2.04&6.09E-10  &  1.99\\
400 & 1.91E-10  &  2.41 &1.80E-10  &  2.24&1.28E-10  &  2.24\\
800 & 5.34E-11  &  1.84 &3.86E-11  &  2.22& 2.76E-11  &  2.21\\
\hline 
\end{tabular} 
}
\caption{\sf Example 2: Same as Table \ref{ex2-tableL1}, but for $\eta=10^{-6}$.\label{ex2-tableL106}}
\end{table}
}
\paragraph{Example 3---Nonlinear Gravitational Potential.} In this example, we consider the system \eref{g11}--\eref{eos1} with the
nonlinear gravitational potentials $\phi(y)=\hf y^2$ and $\phi(y)=\sin(2\pi y)$ subject to the steady-state initial data 
\begin{equation}
\rho(y,0)=e^{-\phi(y)},\quad v(y,0)\equiv0,\quad p(y,0)=e^{-\phi(y)},
\label{icss2}
\end{equation}
(once again, we use discrete versions of these steady states with $L(y)\equiv1$; see Appendix \ref{a1}) with the same boundary conditions as
in Example 2. We first apply both the well-balanced and non-well-balanced CU schemes to this IBVP and compute the solution on a sequence of
different meshes until the final time $T=1$. We observe that the while the well-balanced scheme preserves the steady state \eref{icss2}
within the machine accuracy, the errors in the non-well-balanced computations are of order of the scheme. 

We, next, consider the same IBVP but with the following perturbed initial data:
\begin{equation*}
\rho(y,0)=e^{-\phi(y)},\quad v(y,0)\equiv0,\quad p(y,0)=e^{-\phi(y)}+\eta e^{-100(y-0.5)^2},\quad\eta=10^{-3}.
\end{equation*}
We compute the solution until the final time $T=0.25$ using both the well-balanced and non-well-balanced CU schemes. In Figure \ref{fig:p03}
(left), we plot the pressure perturbations ($p(y,0.25)-e^{-\phi(y)}$) computed using the well-balanced scheme with $N=200$ and $N=2000$
(reference solution) uniform grid cells for $\phi(y)=\hf y^2$. For comparison, we plot the same perturbation computed by applying
non-well-balanced CU scheme with $N=200$ uniform grid points. In Figure \ref{fig:p03} (right), we also include the results obtained by
non-well-balanced CU scheme using a much finer mesh. One can conclude that only the well-balanced scheme can accurately capture the
perturbation on a coarse grid, while a very fine mesh is required to control the perturbation with the non-well-balanced method. 

In Figure \ref{fig:sin_p03}, we demonstrate the pressure perturbation ($p(y,0.25)-e^{-\phi(y)}$) obtained using both the well-balanced and
non-well-balanced CU schemes on $N=200$ uniform grid cells for $\phi(y)=\sin(2\pi y)$. Similarly with the previous discussion, the proposed
well-balanced CU scheme accurately resolves the perturbation on a coarse grid, while the non-well-balanced scheme requires much finer grid
to capture the perturbation as accurately as the well-balanced method does. 
\begin{figure}[ht!]
\centerline{\scalebox{0.45}{\includegraphics[width=1.1\textwidth]{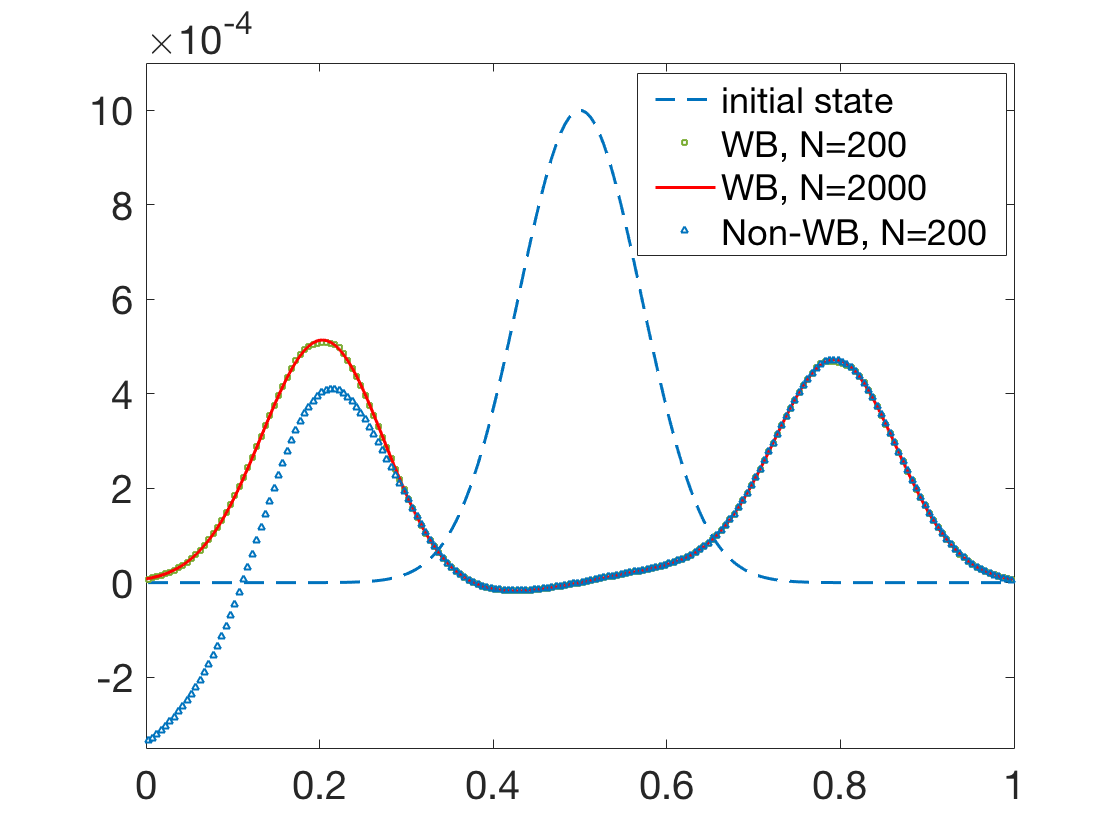}\hspace*{-0.0cm}
{\includegraphics[width=1.1\textwidth]{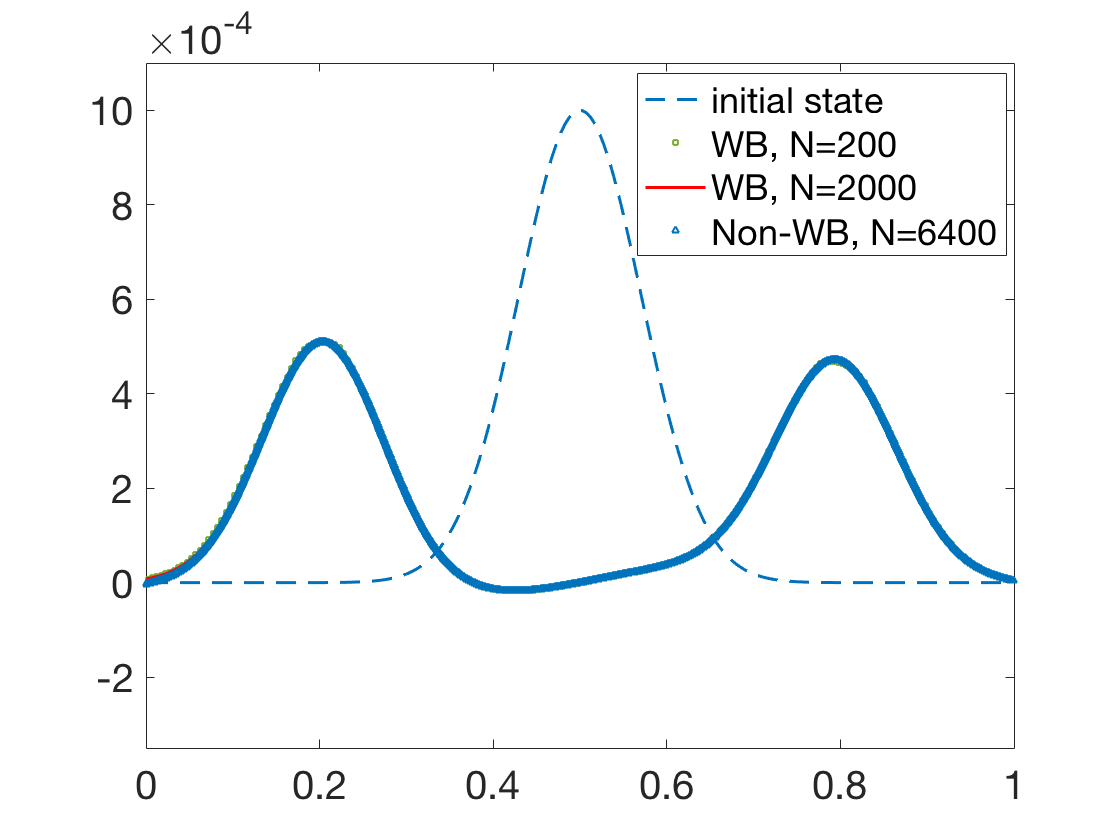}}}}
\caption{\sf Example 3: Pressure perturbation ($p(y,0.25)-e^{-\phi(y)}$) computed by the well-balanced (WB) and non-well-balanced (Non-WB)
CU schemes for $\phi(y)=\hf y^2$ with $N=200$ for each scheme (left) and $N=4000$ for the Non-WB scheme (right). The reference solution is
computed using the WB scheme with $N=2000$ grid points.\label{fig:p03}}
\end{figure}
\begin{figure}[ht!]
\centerline{\scalebox{0.45}{\includegraphics[width=1.1\textwidth]{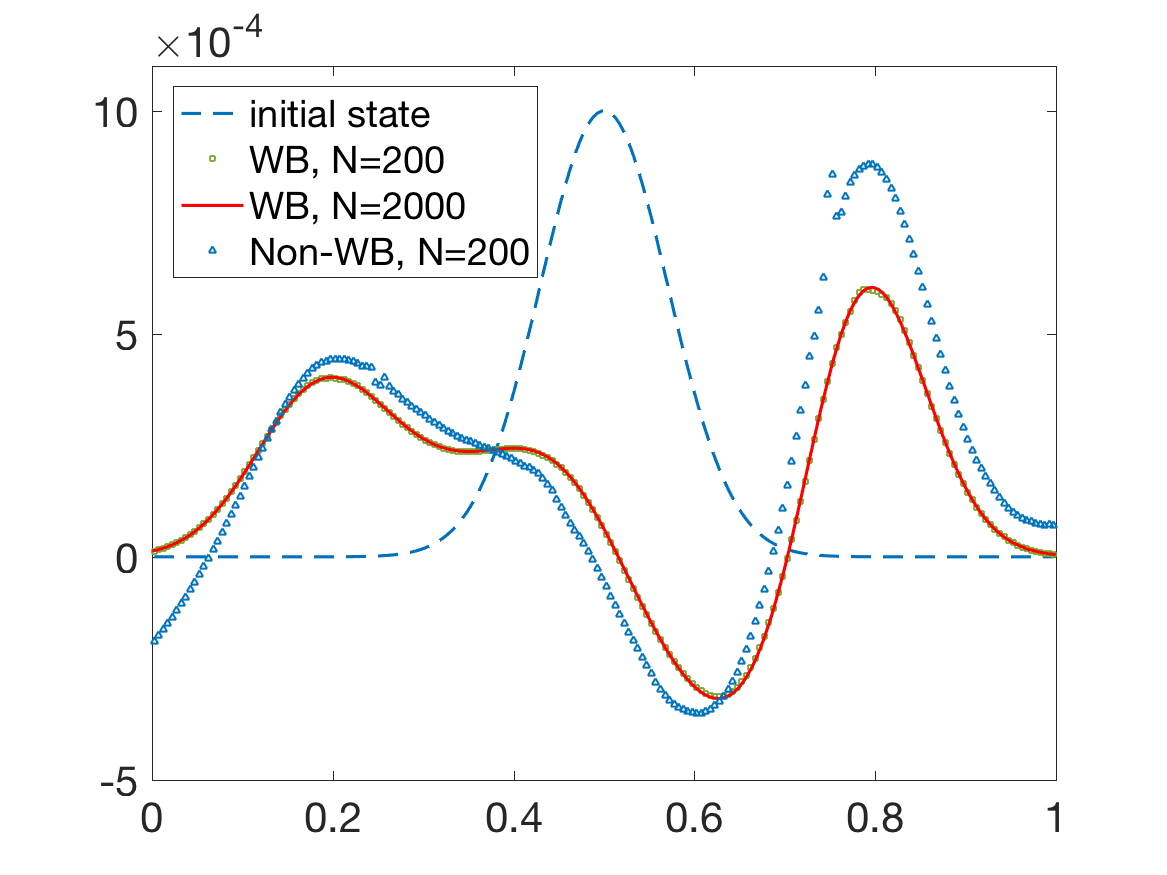}\hspace*{-0.0cm}
{\includegraphics[width=1.1\textwidth]{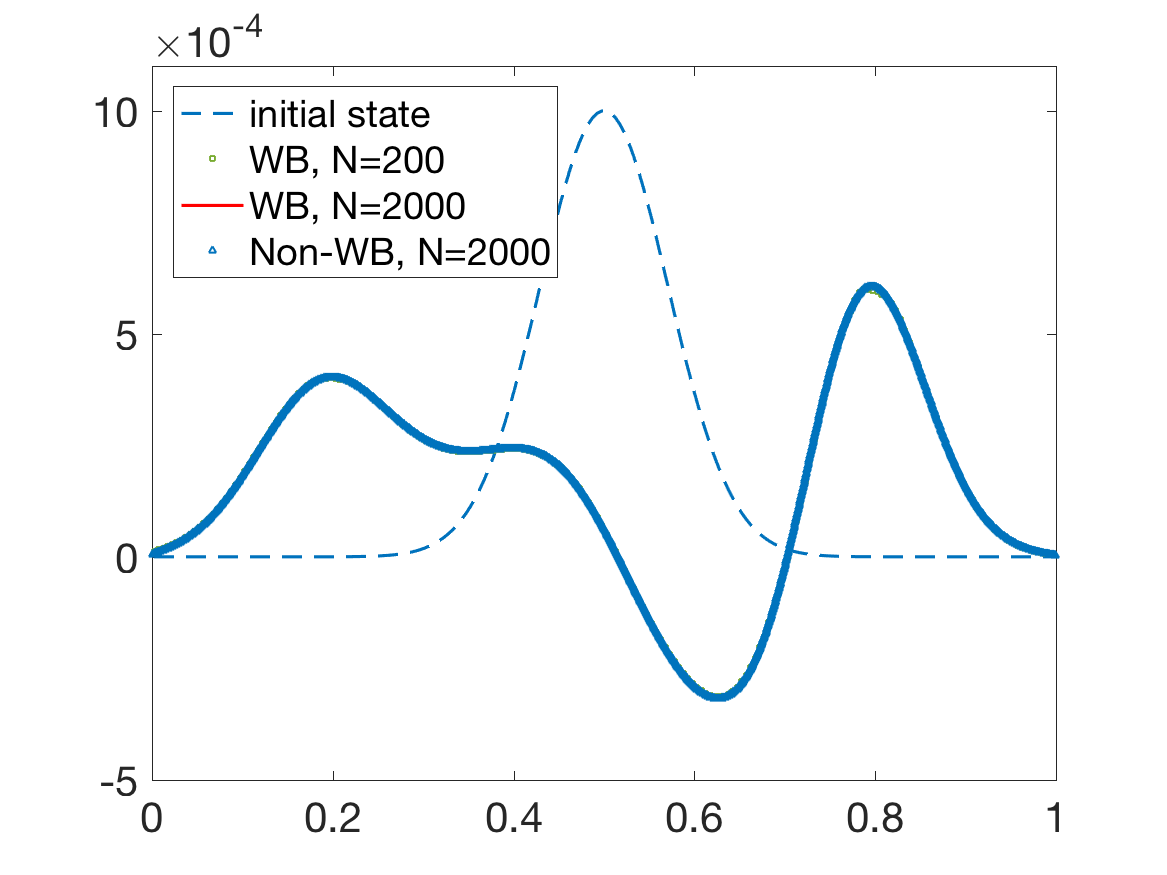}}}}
\caption{\sf Example 3: Pressure perturbation ($p(y,0.25)-e^{-\phi(y)}$) computed by the well-balanced (WB) and non-well-balanced (Non-WB)
CU schemes for $\phi(y)=\sin(2\pi y)$ with $N=200$ for each scheme (left) and $N=2000$ for the Non-WB scheme (right). The reference solution
is computed using the WB scheme with $N=2000$ grid points.\label{fig:sin_p03}}
\end{figure}

\subsection{Two-Dimensional Examples}\label{sec42}
{\color{black} In this section, we test the perfomance of the proposed well-balanced central-upwind scheme on two numerical examples and 
compare it with the performance of the corresponding non-well-balanced scheme. We note that the computational cost of the well-balanced 
scheme is in average about 25--40\% higher than of its non-well-balanced couterpart when mesuared on the same grid. However, the resolution
achieved by the well-balanced scheme is much higher. This is especially pronounced when course grd computations are conducted. In such 
cases, in order to achieve a comparabale resolution by the non-well-balanced scheme, the computations must be perfomed on a much finer 
grid, which  makes the non-well-balanced scheme significantly less efficient as demonstrated in Example 4.}

\paragraph{Example 4---Isothermal Equilibrium Solution.}
In the first 2-D example, which was introduced in \cite{XinShu13}, we consider the system \eref{eq:global} with $\phi(x,y)=x+y$ subject to
the following initial data:
\begin{equation}
\rho(x,y,0)=1.21e^{-1.21\phi(x,y)},\quad u(x,y,0)\equiv v(x,y,0)\equiv0,\quad p(x,y,0)=e^{-1.21\phi(x,y)},
\label{4.2}
\end{equation}
satisfying \eref{1.11} and impose zero-order extensions at all of the four edges of the unit square $[0,1]\times[0,1]$. As in the 1-D case,
we use a discrete version of \eref{4.2} described in Appendix \ref{a2} rather than its continuous counterpart (here, $L(x,y,0)=e^{-1.21x}$
and $K(x,y,0)=e^{-1.21y}$).

We first use the discrete steady-state data and verify that they are preserved within the machine accuracy, when the solution is computed by
the proposed well-balanced CU scheme. On contrary, the non-well-balanced CU scheme preserves the initial equilibrium within the accuracy of 
the scheme only. 

Next, we add a small perturbation to the initial pressure and replace $p(x,y,0)$ in \eref{4.2} with
\begin{equation*}
p(x,y,0)=e^{-1.21\phi(x,y)}+\eta e^{-121\left((x-0.3)^2+(y-0.3)^2\right)},\quad\eta=10^{-6}.
\end{equation*}
In Figure \ref{fig49} and the upper row of Figure \ref{fig410}, we plot the pressure perturbation computed by both the well-balanced and
non-well-balanced CU schemes at time $T=0.15$ using $100\times100$ uniform cells. As one can clearly see, the well-balanced CU scheme can
capture the perturbation accurately {\color{black} (and it takes only in 2.18 seconds on a laptop)}, while the non-well-balanced one 
produces spurious  waves. When the mesh is refined to $800\times800$ uniform cells, the well-balanced solution remains oscillation-free; 
see Figure \ref{fig410} (lower left), and the spurious waves appearing in the non-well-balanced solution disappear; see
Figure \ref{fig410} (lower right). {\color{black} We stress that the non-well-balanced $800\times800$ computation takes  $1180.80$ seconds on
the same laptop. One can conclude that while the well-balanced CU scheme captures the perturbation efficiently on a coarse grid, the 
non-well-balanced CU scheme consumes very long time to achieve similar results.}
\begin{figure}[ht!]
\centerline{\scalebox{0.41}{\includegraphics{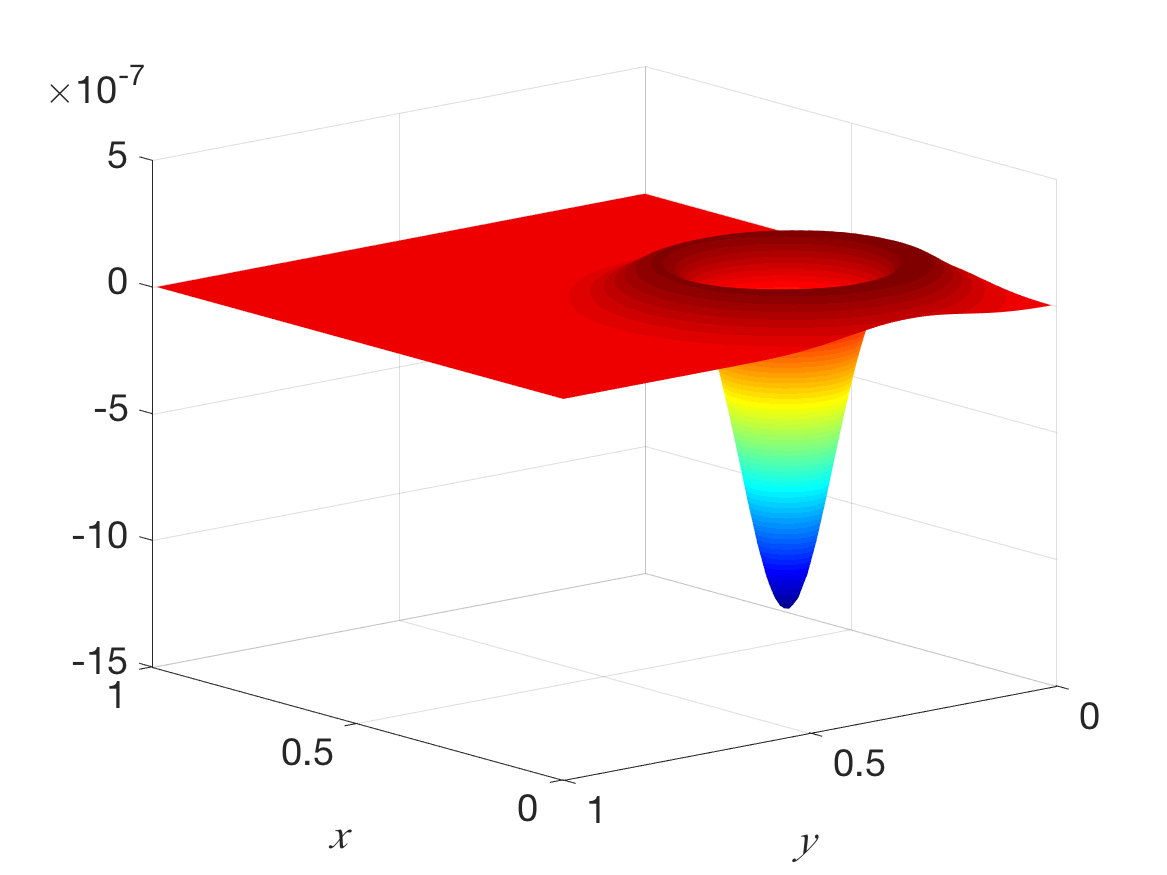}\hspace*{-0.0cm}{\includegraphics{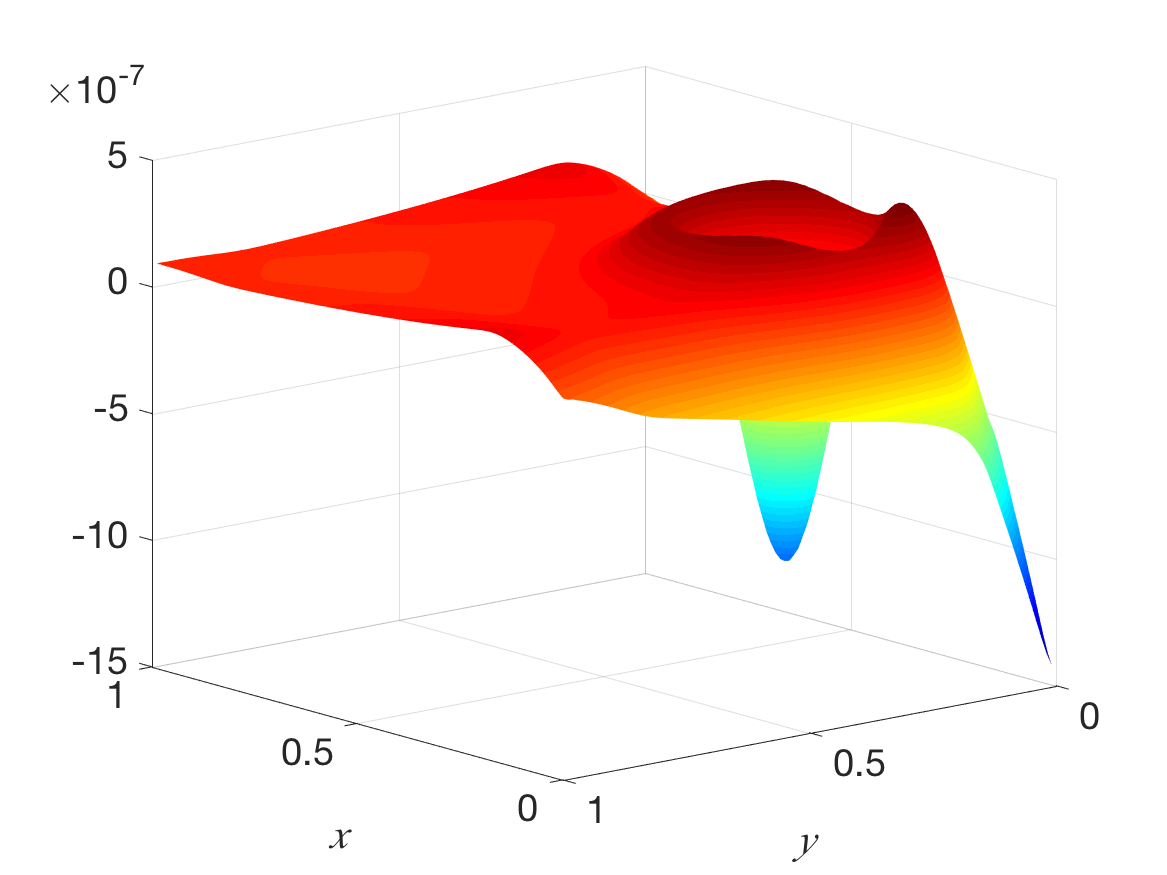}}}}
\caption{\sf Example 4: Pressure perturbation computed by the well-balanced (left) and non-well-balanced (right) CU schemes using
$100\times100$ uniform cells.\label{fig49}}
\end{figure}
\begin{figure}[ht!]
\centerline{\scalebox{0.45}{\includegraphics{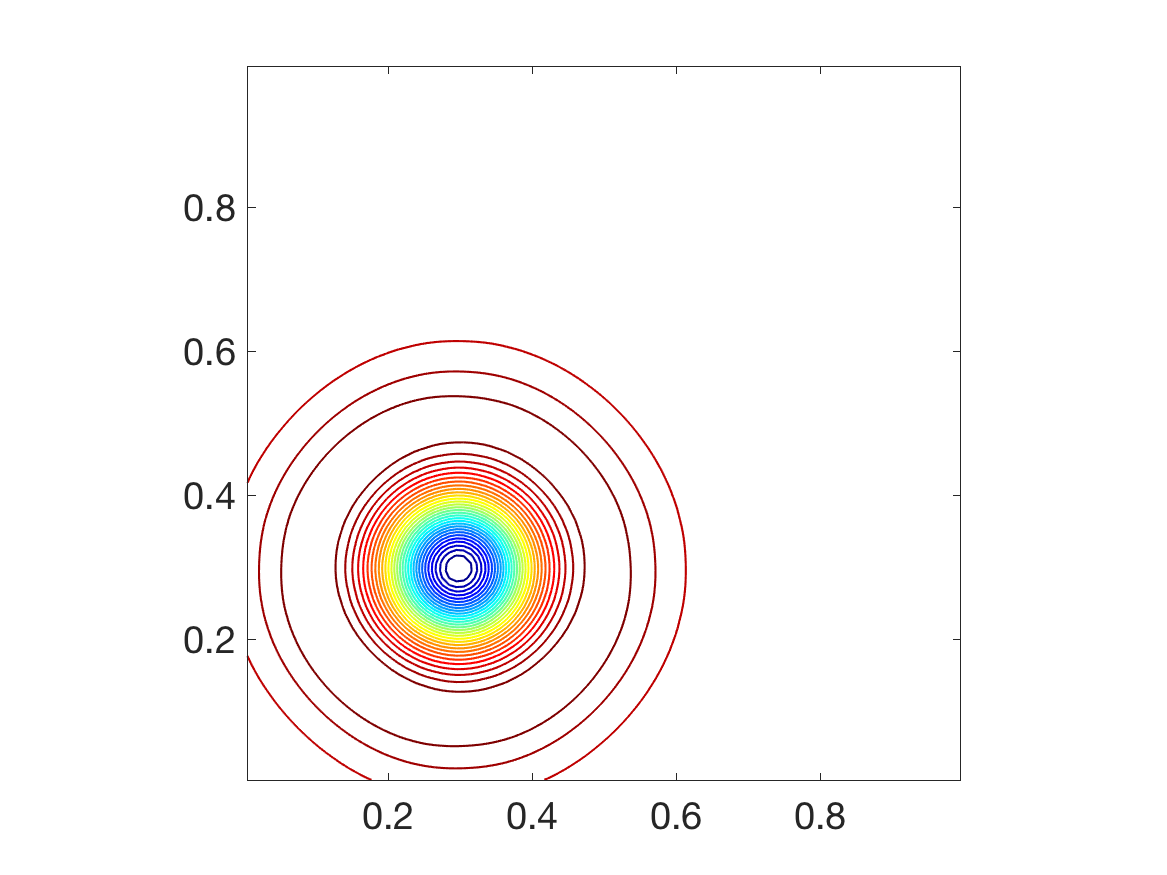}\hspace*{-2.5cm}{\includegraphics{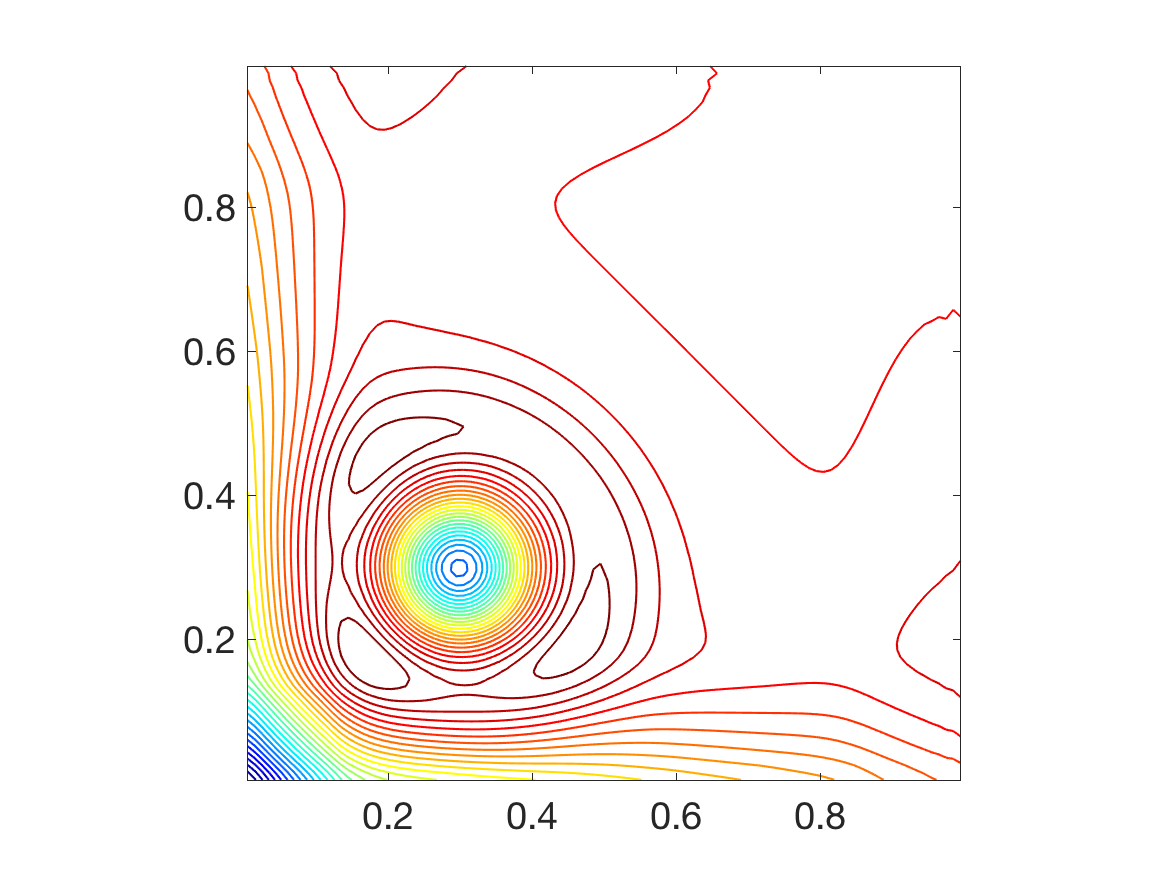}}}}
\centerline{\scalebox{0.45}{\includegraphics{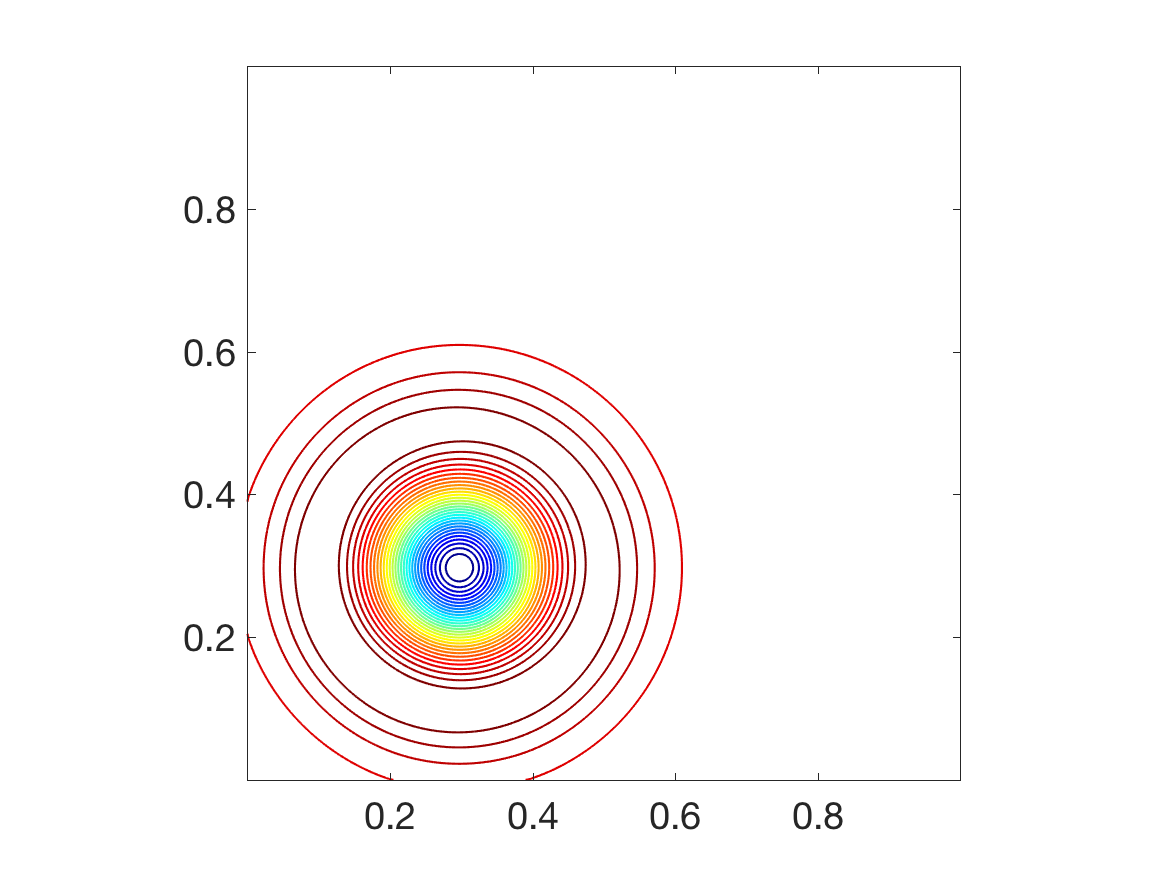}\hspace*{-2.5cm}{\includegraphics{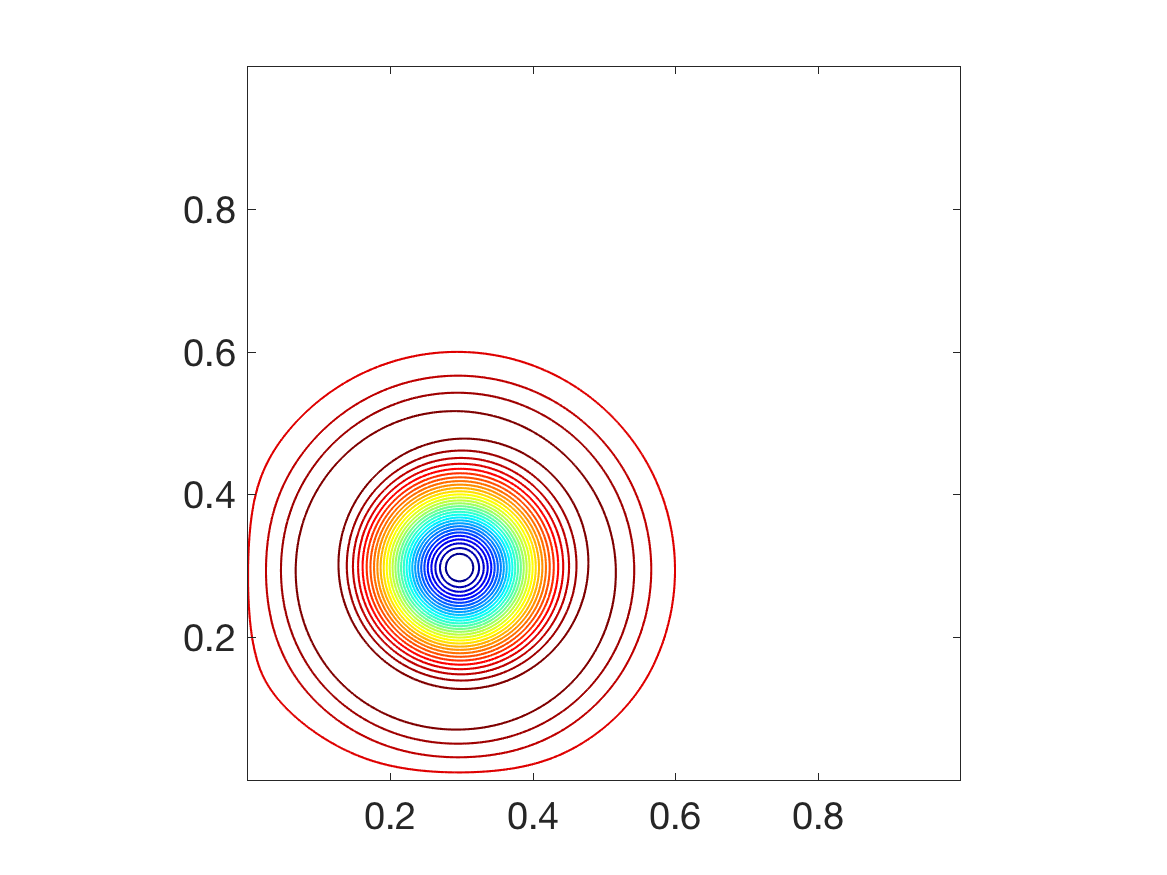}}}}
\caption{\sf Example 4: Contour plot of the pressure perturbation computed by well-balanced (left column) and non-well-balanced (right
column) CU schemes using $100\times100$ (upper row) and $800\times800$ (lower row) uniform cells.\label{fig410}}
\end{figure}
	
\paragraph{Example 5---Explosion.}
In the second 2-D example, we compare the performance of well-balanced and non-well-balanced CU schemes in an explosion setting and
demonstrate nonphysical shock waves generated by non-well-balanced scheme.
	
We solve the system \eref{eq:global} with $\phi(x,y)=0.118y$ in the computational domain $[0,3]\times[0,3]$ subject to the following initial
data:
$$
\begin{aligned}
&\rho(x,y,0)\equiv1,\quad u(x,y,0)\equiv v(x,y,0)\equiv0,\\
&p(x,y,0)=1-\phi(x,y)+\left\{\begin{array}{ll}0.005,&(x-1.5)^2+(y-1.5)^2<0.01,\\0,&~\mbox{otherwise}.\end{array}\right.
\end{aligned}
$$ 
Zero-order extrapolation is used as the boundary conditions in all of the directions.
	
We use a uniform grid with $101\times101$ cells and compute the solution by both the well-balanced and non-well-balanced CU schemes until
the final time $T=2.4$. At first, a circular shock wave is developed and later on it transmits through the boundary. Due to the heat
generated by the explosion, the gas at the center expands and its density decreases generating a positive vertical momentum at the center of
the domain. In Figures \ref{Ex3wb} and \ref{Ex3nwb}, we plot the solution ($\rho$ and $\sqrt{u^2+v^2}$ at times $t=1.2$, 1.8 and 2.4)
computed by the well-balanced and non-well-balanced schemes, respectively. As one can see, the well-balanced scheme accurately captures the
behavior of the solution at all stages, while the non-well-balanced scheme produces significant oscillations at the smaller time $t=1.2$,
which starts dominating the solution, especially its velocity field, by the final time $T=2.4$.
\begin{figure}[ht!]
\centerline{\hspace{-0.0cm}
\includegraphics[trim=120 20 100 0, clip, width=0.29\textwidth]{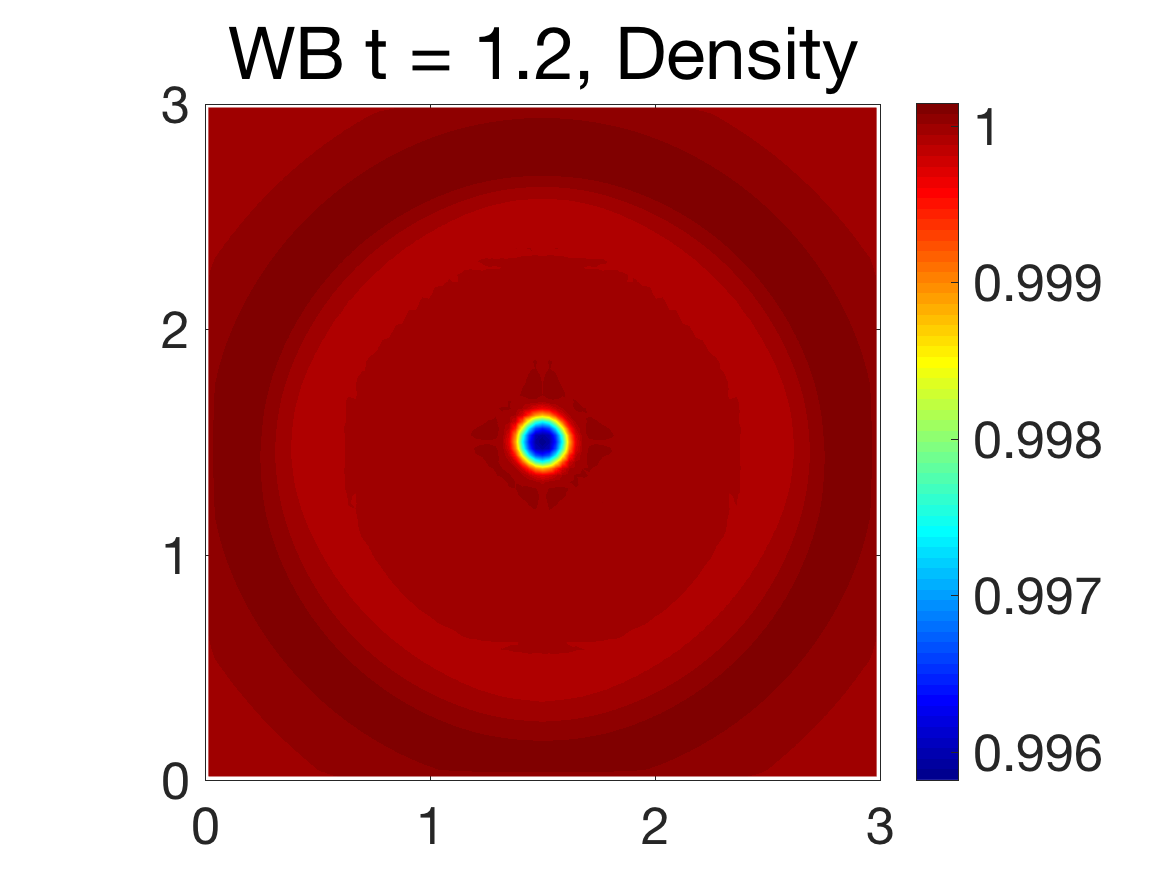}\hspace{0.2cm}
\includegraphics[trim=120 20 100 0, clip, width=0.29\textwidth]{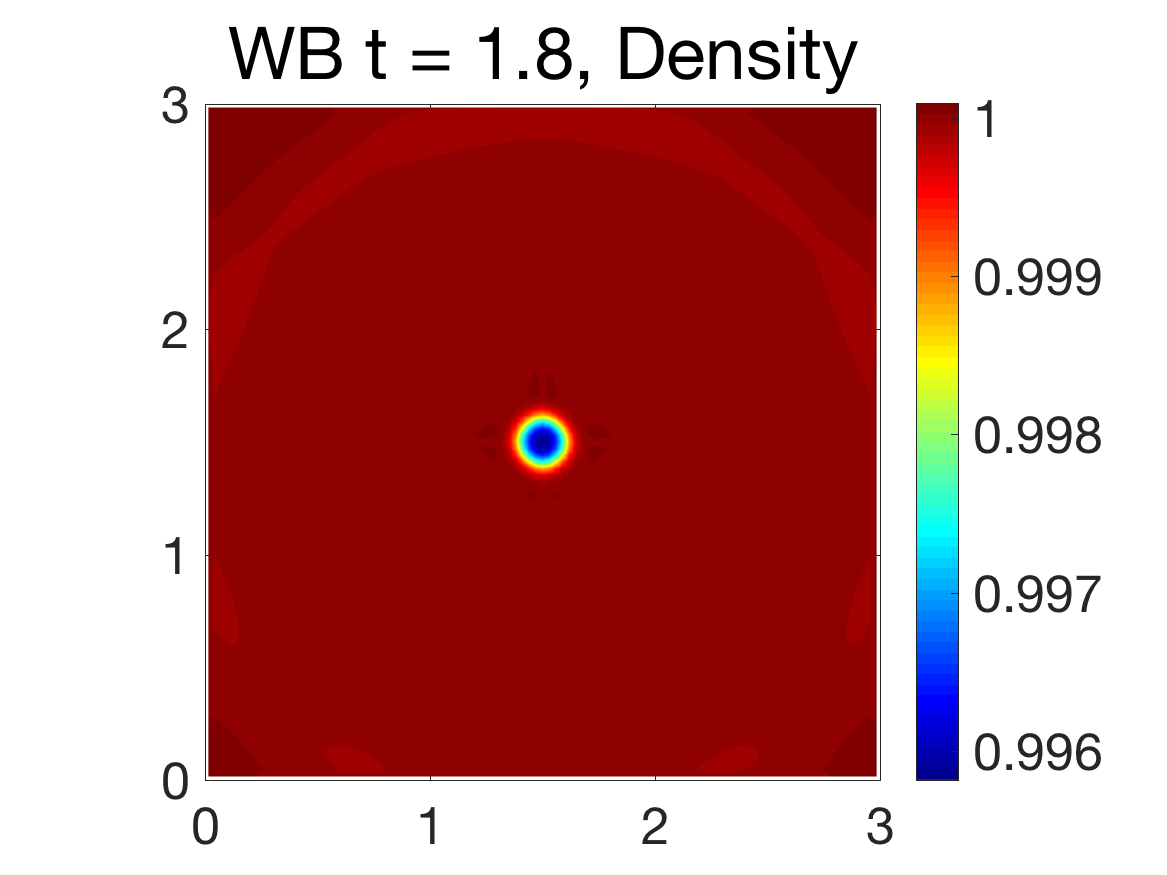}\hspace{-0.0cm}
\includegraphics[trim=80 20 30 0, clip, width=0.384\textwidth]{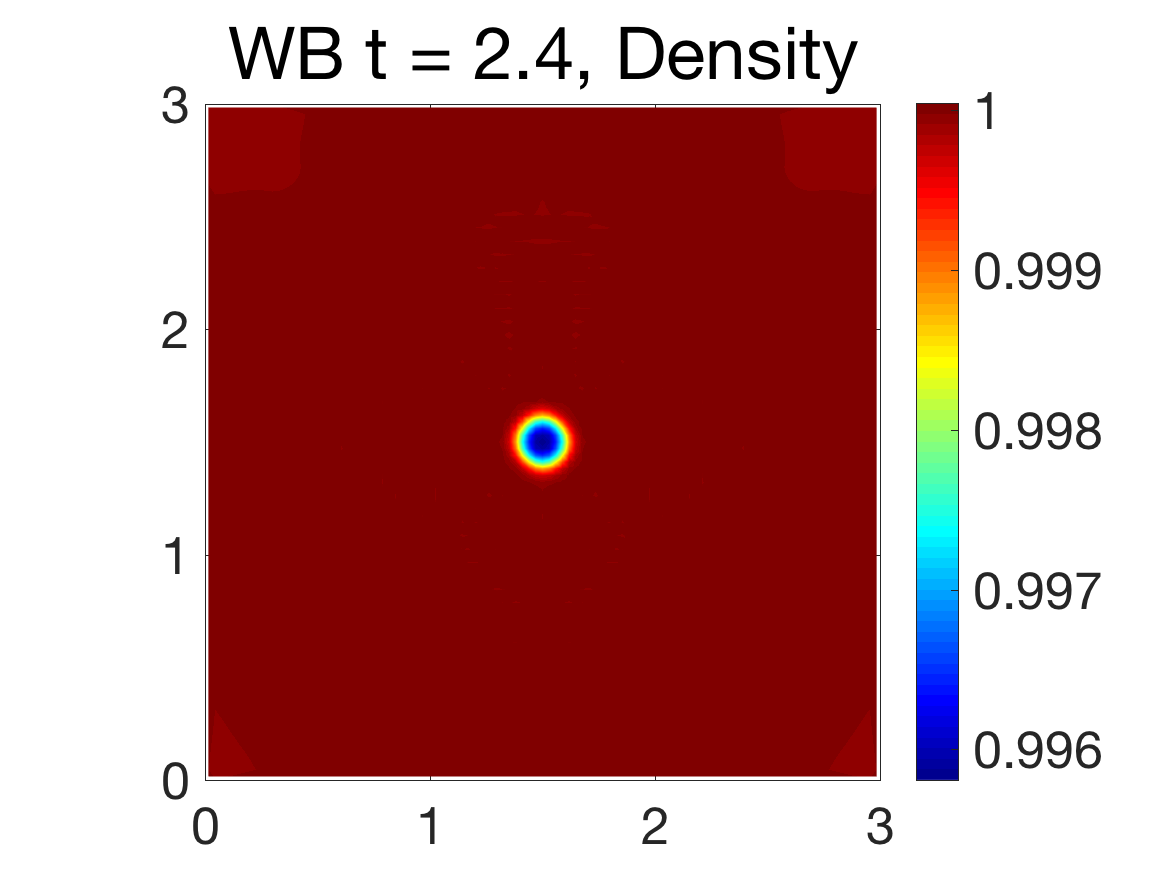}
}
\centerline{\hspace{-0.0cm}
\includegraphics[trim=120 40 100 0, clip, width=0.29\textwidth]{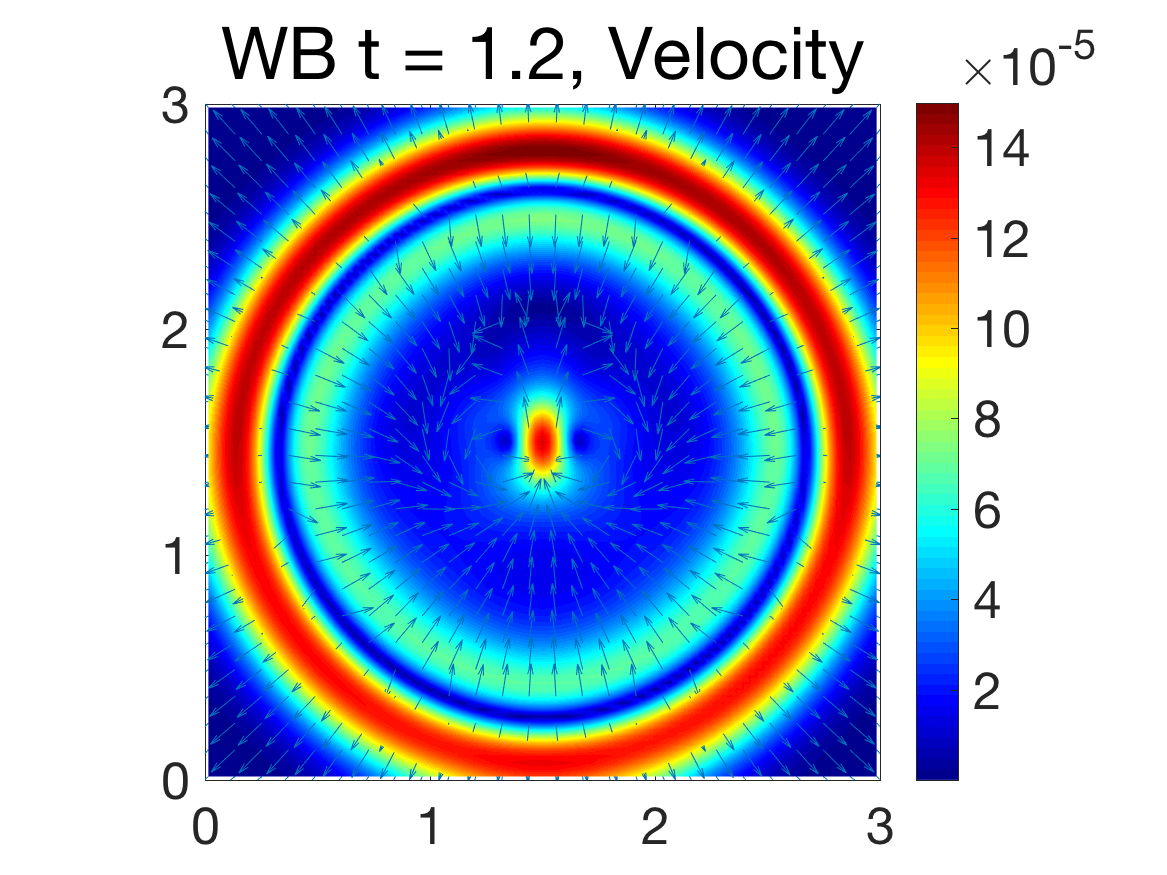}\hspace{0.2cm}
\includegraphics[trim=120 40 100 0, clip, width=0.29\textwidth]{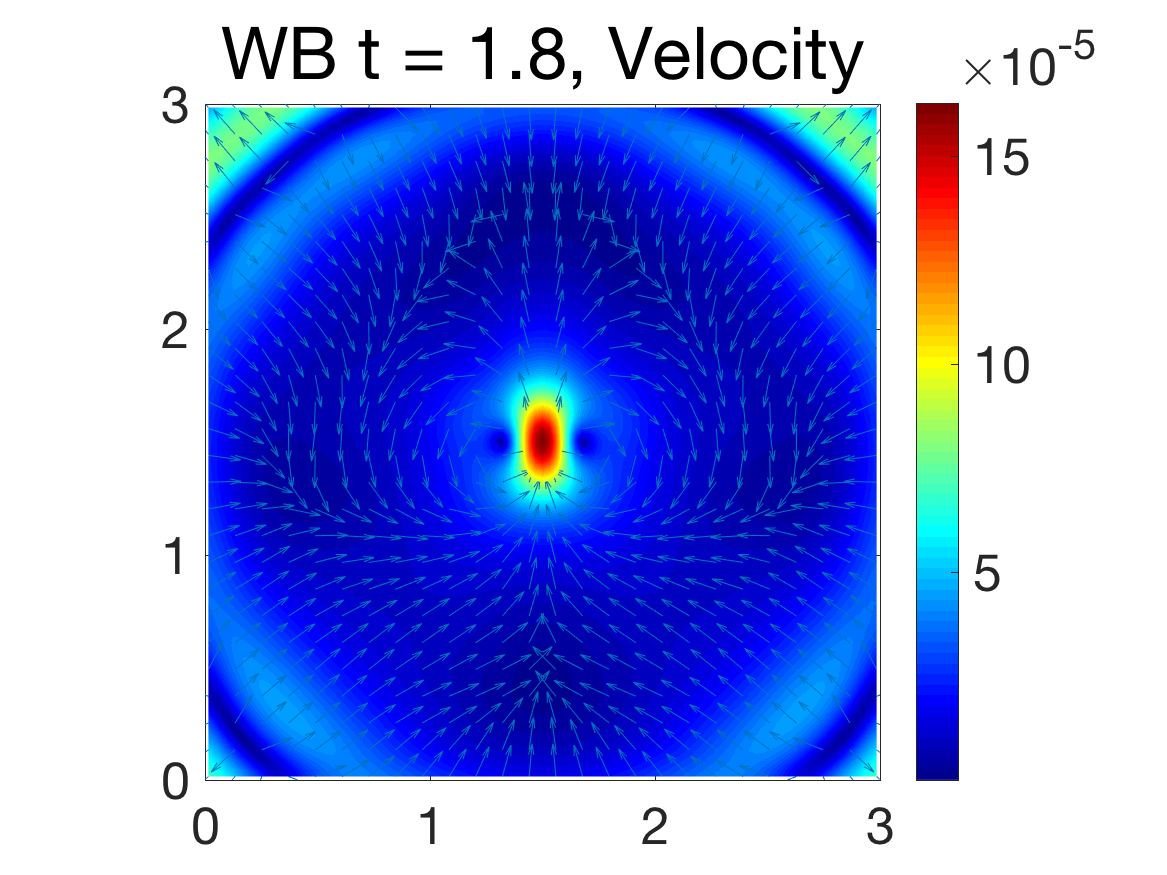}\hspace{-0.0cm}
\includegraphics[trim=80 40 30 0, clip, width=0.384\textwidth]{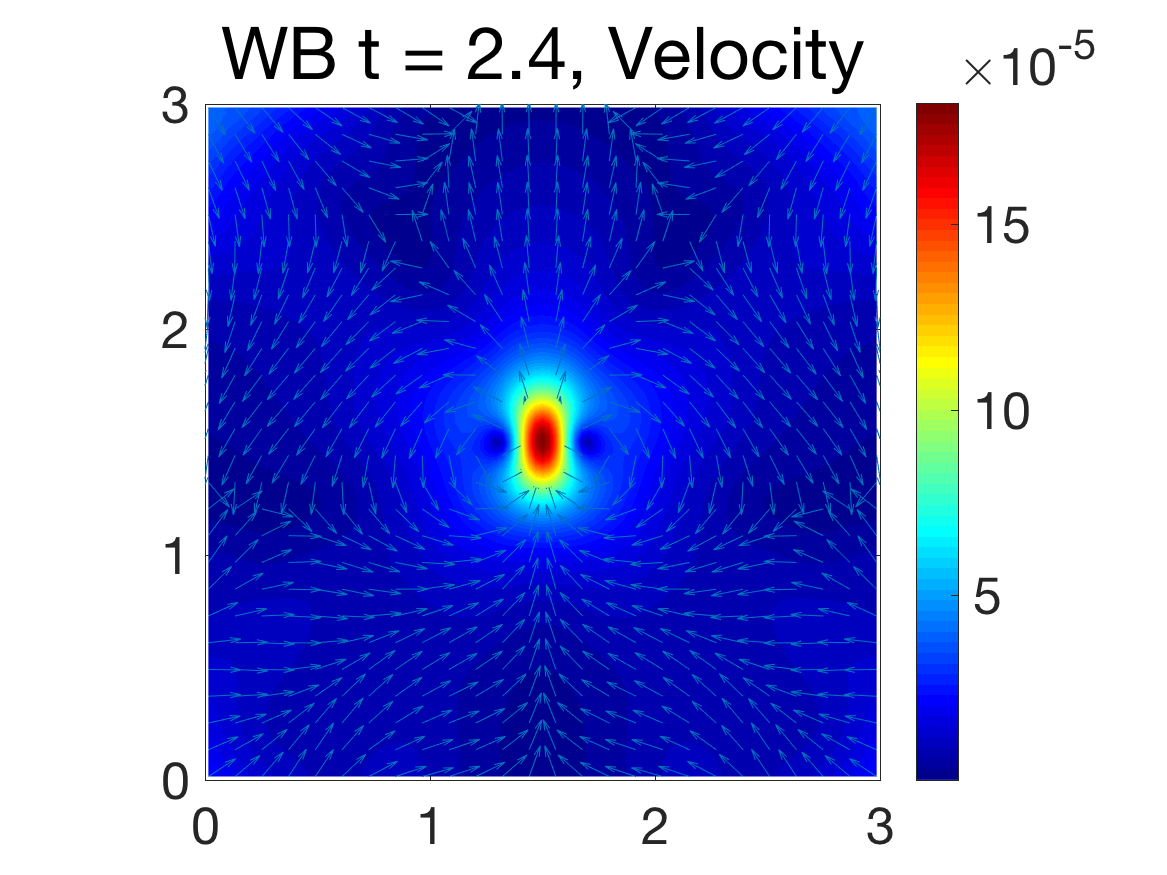}
}
\caption{\sf Example 5: Density ($\rho$) and velocity ($\sqrt{u^2+v^2}$) computed by the well-balanced CU scheme in the domain
$[0,3]\times[0,3]$.\label{Ex3wb}}
\end{figure}
\begin{figure}[ht!]
\centerline{\hspace{-0.0cm}
\includegraphics[trim=120 20 100 0, clip, width=0.29\textwidth]{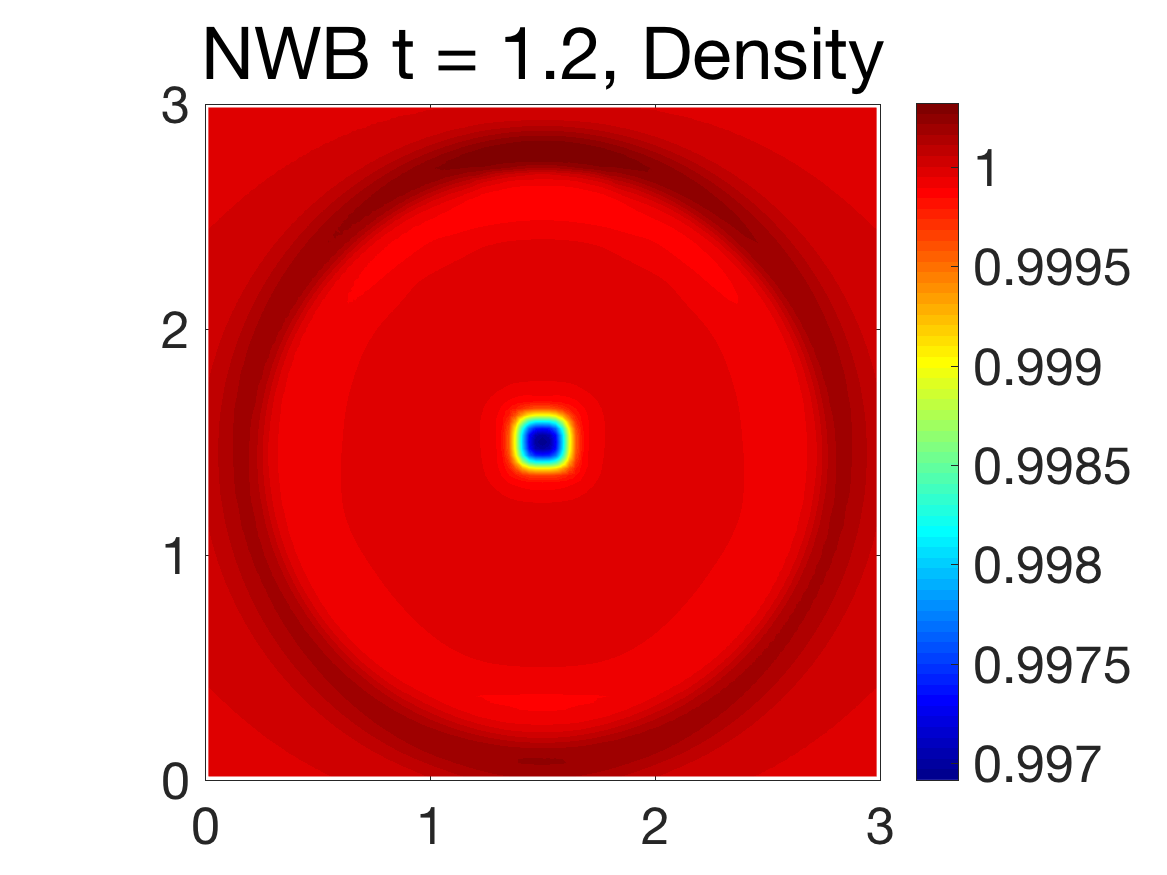}\hspace{0.2cm}
\includegraphics[trim=120 20 100 0, clip, width=0.29\textwidth]{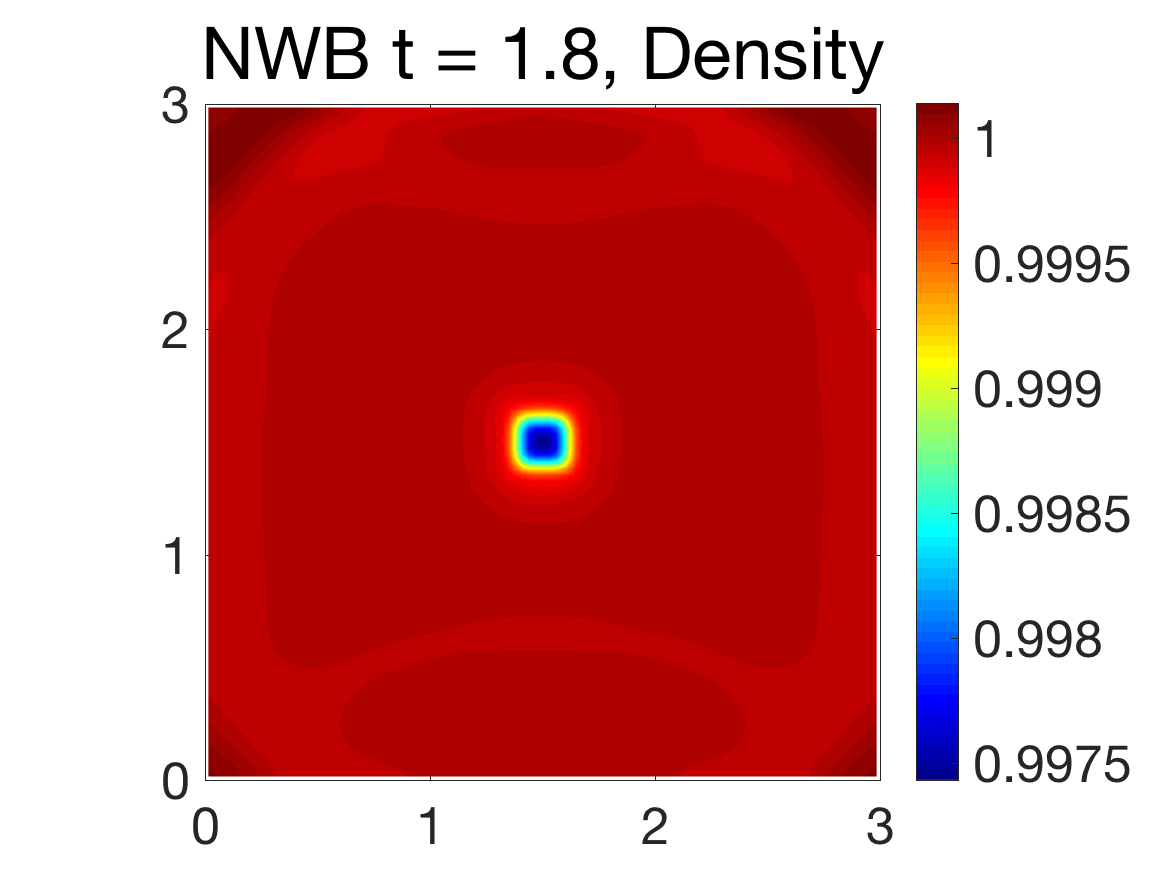}\hspace{-0.0cm}
\includegraphics[trim=80 20 30 0, clip, width=0.384\textwidth]{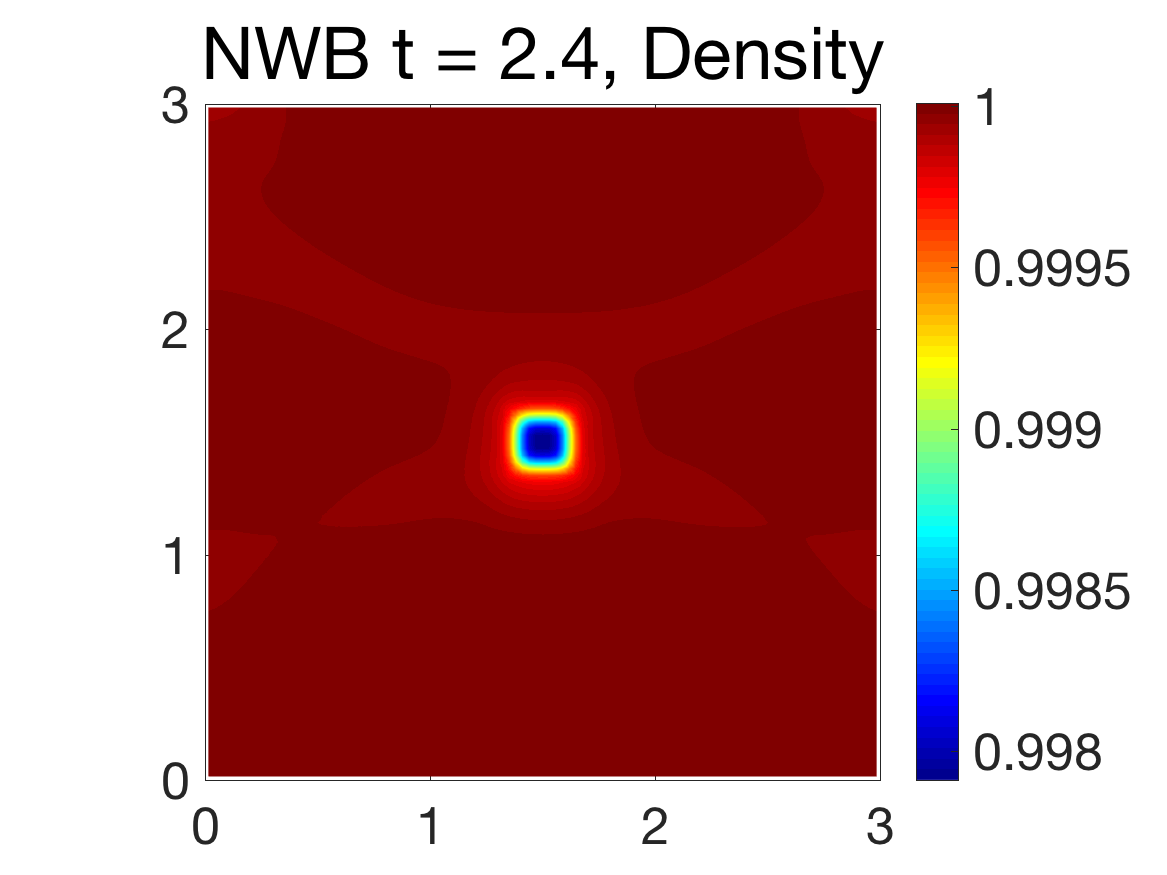}
}
\centerline{\hspace{-0.0cm}
\includegraphics[trim=120 40 100 0, clip, width=0.29\textwidth]{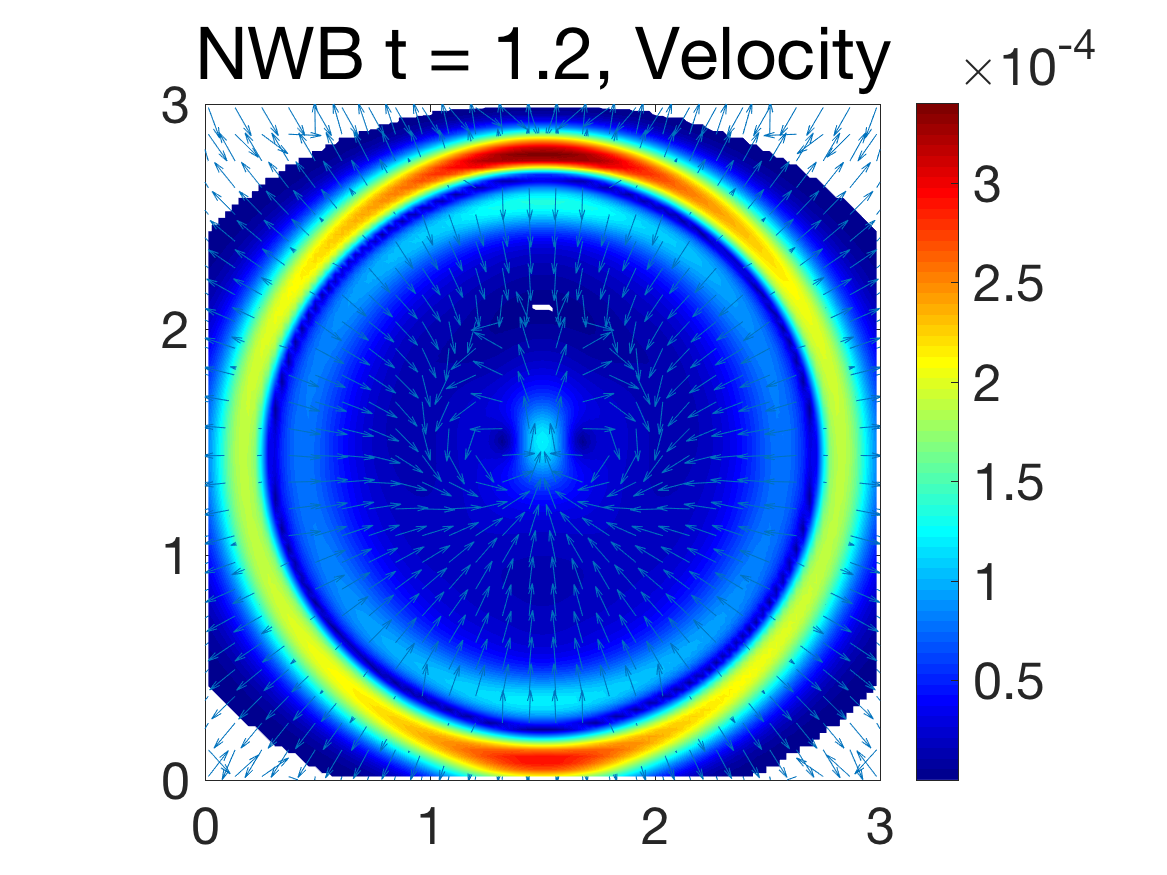}\hspace{0.2cm}
\includegraphics[trim=120 40 100 0, clip, width=0.29\textwidth]{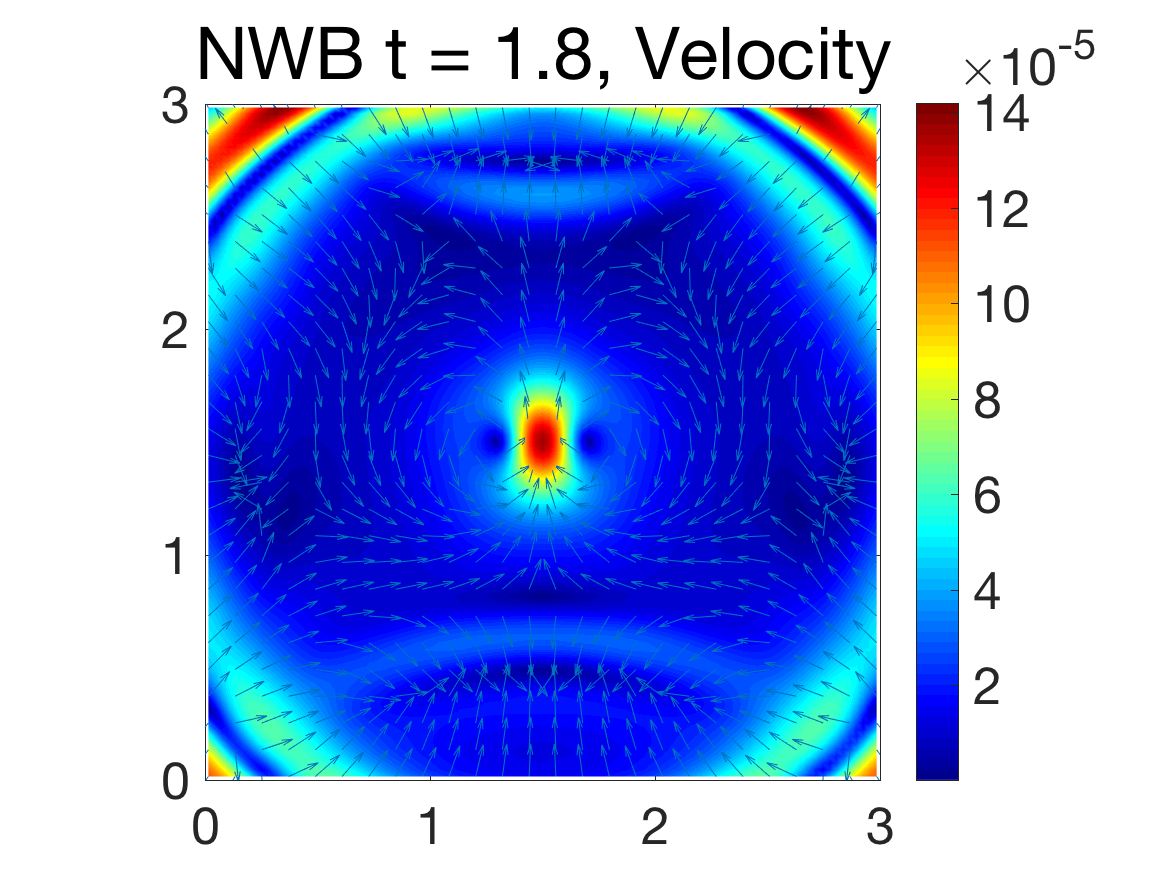}\hspace{-0.0cm}
\includegraphics[trim=80 40 30 0, clip, width=0.384\textwidth]{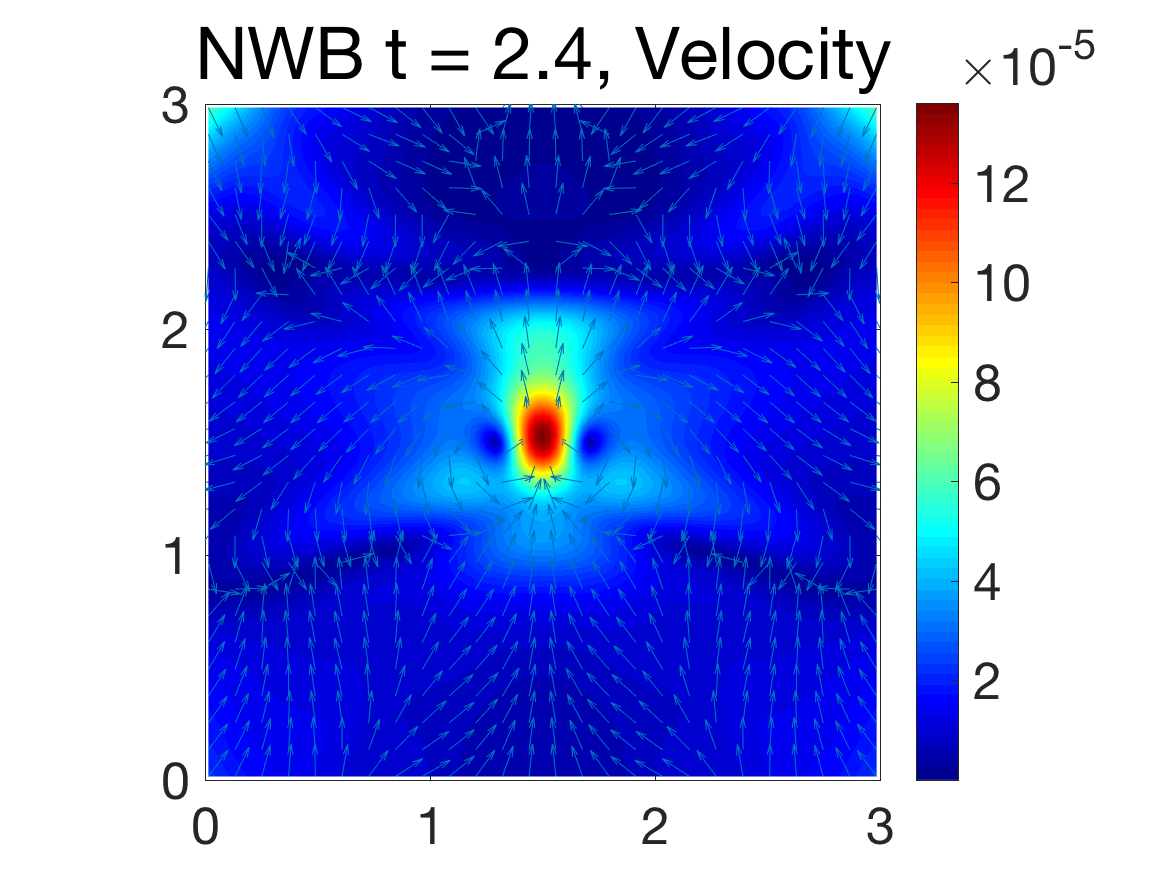}
}
\caption{\sf Example 5: Density ($\rho$) and velocity ($\sqrt{u^2+v^2}$) computed by the non-well-balanced CU scheme in the domain
$[0,3]\times[0,3]$.\label{Ex3nwb}}
\end{figure}

{\color{black}
\section{Conclusion}\label{sec5}
We have presented a new well-balanced CU scheme for the Euler equations with gravitation. Our scheme is based on a new conservative
reformulation of the original system of balance laws. This reformulation is obtained by introducing new global equilibrium variables which
are also used in the reconstruction step and together with a well-balanced time evolution ensure the well-balanced property of the resulting
numerical method. 

It should be observed that in some astrophysical applications the gravitational energy $\rho\phi$ may be orders of magnitude larger than the
fluid energy $E$. We have not considered such cases in the current paper, but would like to point out that combining relatively small fluid 
dynamics quantities $p$ and $E$ with larger gravitational quantities $Q$, $R$ and $\rho\phi$, may, in principle, lead to inaccurate 
calculations of point values of $p$ and $E$ from the reconstructed values of new variables $K=p+Q$, $L=p+R$ and $E+\rho\phi$. The 
applicability of the proposed well-balanced approach in these special cases should be studied separately.

We would also like to note that in the studied system \eref{1.2} and its 1-D version the gravitational potential $\phi$ is assumed to be
time-independent. In the case of time-dependent $\phi$ the conservative reformulations \eref{eq:global}, \eref{QR} and \eref{g11},
\eref{g12} are not valid. However, one can still introduce the global fluxes in the momenta equations, while treating the source term in the
energy equation using the midpoint quadrature, which will lead to a well-balanced central-upwind scheme.}

\begin{acknowledgment}
The work of A. Chertock was supported in part by NSF grant DMS-1521051. The work of A. Kurganov was supported in part by NSF grant
DMS-1521009. The work of \c{S}. N.~\"{O}zcan was supported in part by the Turkish Ministry of National Education. The work of E. Tadmor was
supported in part by ONR grant N00014-1512094 and NSF grants RNMS11-07444 (Ki-Net) and DMS-1613911.
\end{acknowledgment}

{\color{black}
\appendix
\section{Proof of Theorem \ref{thm3.1}}\label{prf}
Assume that at certain time level, we have
\begin{equation}
u_{j,k}^{\rm E}=u_{j,k}^{\rm W}=u_{j,k}^{\rm N}=u_{j,k}^{\rm S}=v_{j,k}^{\rm E}=v_{j,k}^{\rm W}=v_{j,k}^{\rm N}=v_{j,k}^{\rm S}\equiv0
\label{data1}
\end{equation}
and
\begin{equation}
K_{j,k}^{\rm E}=K_{j,k}^{\rm W}=\widehat K_k,~~\forall j,\qquad L_{j,k}^{\rm N}=L_{j,k}^{\rm S}=\widehat L_j,~~\forall k,
\label{data2}
\end{equation}
where $\widehat K_k$ and $\widehat L_j$ only depend on $k$ and $j$, respectively. In order to prove that the proposed scheme is
well-balanced, we will show that for the data in \eref{data1} and \eref{data2} the $x$-numerical fluxes $\bm{\mathcal {F}}_{\jph,k}$ depend
on $k$ only, while the $y$-numerical fluxes $\bm{\mathcal {G}}_{j,\kph}$ depend on $j$ only. This will ensure that the RHS of \eref{3.10}
is identically equal to zero at such steady states.

Indeed, the first components of the numerical fluxes, $\mathcal{F}^{(1)}_{\jph,k}$ and $\mathcal{G}^{(1)}_{j,\kph}$, vanish since
\eref{data1} is satisfied and \eref{data2} implies $H(\psi_{\jph,k})=H(\psi_{j,\kph})=H(0)=0$. The second component of the $x$-numerical
flux is $\mathcal{F}^{(2)}_{\jph,k}=\widehat K_k$ since $u_{j,k}^{\rm E}=u_{j+1,k}^{\rm W}=0$ and
$K_{j,k}^{\rm E}=K_{j+1,k}^{\rm W}=\widehat K_k$. Similarly, the third component of the $y$-numerical flux is
$\mathcal{G}^{(3)}_{j,\kph}=\widehat L_j$ since $v_{j,k}^{\rm N}=u_{j,k+1}^{\rm S}=0$ and $L_{j,k}^{\rm N}=L_{j,k+1}^{\rm S}=\widehat L_j$.
Next, \eref{data1} implies that the third component of the $x$-numerical flux and the second component of the $y$-numerical flux,
$\mathcal{F}^{(3)}_{\jph,k}$ and $\mathcal{G}^{(2)}_{j,\kph}$, vanish. Finally, the fourth component of the $x$-numerical flux vanishes:
\begin{equation*}
\begin{aligned}
\mathcal{F}^{(4)}_{\jph,k}&=\alpha_{\jph,k}\left[E_{j+1,k}^{\rm W}-E_{j,k}^{\rm E}+H(\psi_{\jph,k})\cdot
\left((\rho\phi)_{j+1,k}^{\rm W}-(\rho\phi)_{j,k}^{\rm E}-\delta(E+\rho\phi)_{\jph,k}\right)\right]\\
&=\frac{\alpha_{\jph,k}}{\gamma-1}\cdot\frac{p_{j+1,k}^{\rm W}-p_{j,k}^{\rm E}}{2}
=\frac{\alpha_{\jph,k}}{2(\gamma-1)}\left[\big(K_{j+1,k}^{\rm W}-Q_{\jph,k})-(K_{j,k}^{\rm E}-Q_{\jph,k}\big)\right]=0,
\end{aligned}
\end{equation*}
since $K_{j,k}^{\rm E}=K_{j+1,k}^{\rm W}=\widehat K_k$ for all $j$. Similarly, the fourth component of the $y$-numerical flux also
vanishes:
\begin{equation*}
\begin{aligned}
\mathcal{G}^{(4)}_{j,\kph}&=\beta_{j,\kph}\left[E_{j,k+1}^{\rm S}-E_{j,k}^{\rm N}+H(\psi_{j,\kph})\cdot
\left((\rho\phi)_{j,k+1}^{\rm S}-(\rho\phi)_{j,k}^{\rm N}-\delta(E+\rho\phi)_{j,\kph}\right)\right]\\
&=\frac{\beta_{j,\kph}}{\gamma-1}\cdot\frac{p_{j,k+1}^{\rm S}-p_{j,k}^{\rm N}}{2}
=\frac{\beta_{j,\kph}}{2(\gamma-1)}\left[\big(L_{j,k+1}^{\rm S}-R_{j,\kph})-(L_{j,k}^{\rm N}-R_{j,\kph}\big)\right]=0,
\end{aligned}
\end{equation*}
since $L_{j,k}^{\rm N}=L_{j,k+1}^{\rm S}=\widehat L_j$ for all $k$.$\hfill\blacksquare$
}

\section{Discrete steady states}\label{app1}
In this section, we describe a possible way discrete steady states can be constructed in both the 1-D and 2-D cases.

\subsection{One-Dimensional Discrete Steady States}\label{a1}
The initial data corresponding to a 1-D discrete steady state can be constructed as follows. Given
$\rho_k:=\rho(y_k),~v_k:=0,~L_k=L(y_k)={\rm Const},~k=k_\ell,\ldots,k_r$, we use the recursive formula \eref{RRR} to obtain the discrete
values $R_k$. Then, $p_k$ are computed from \eref{2.15}, $p_k=L_k-R_k,~k=k_\ell,\ldots,k_r$, and $E_k$ are obtained from the EOS,
$E_k=\frac{p_k}{\gamma-1}+\frac{1}{2}\rho_k(v_k)^2,~k=k_\ell,\ldots,k_r$.

\subsection{Two-Dimensional Discrete Steady States}\label{a2}
The initial data corresponding to a 2-D discrete steady state can be constructed in a similar though more complicated way. In the 2-D case,
both $u_{j,k}=v_{j,k}=0$, but neither $L$ nor $K$ are constant throughout the computational domain. However, since $L$ is independent of $y$
and $K$ is independent of $x$, we have $L_{j,k}=L_j$ and $K_{j,k}=K_k,~j=j_\ell,\ldots,j_r,~k=k_\ell,\ldots,k_r$, which are assumed to be
given.

We now show how one can construct a steady-state solution for a specific case of $\phi(x,y)=\phi(x+y)$, which is satisfied by $\phi$ used in
Example 4.

We first set the zero values of $Q$ and $R$ (introduced in \eref{QR}) at the right and bottom parts of the boundary, respectively
(alternatively, one could set the zero values of $R$ and $S$ at the left and upper parts):
$$
Q_{j_{r}+\hf,k}=0,~~k=k_\ell,\ldots,k_r,\quad R_{j,k_\ell-\hf}=0,~~j=j_\ell,\ldots,j_r.
$$
We then obtain the discrete values of $\rho_{j,k}$ and $p_{j,k}$ recursively as $j=j_r,\ldots,j_\ell$ and $k=k_\ell,\ldots,k_r$:
\begin{equation*}
\begin{aligned}
&\rho_{j,k}=\frac{L_j-K_k-R_{j,\kmh}+Q_{\jph,k}}{\frac{\dx}{2}\phi_x+\frac{\dy}{2}\phi_y},\\
&p_{j,k}=\frac{L_j+K_k-R_{j,\kmh}-Q_{\jph,k}+\left(\frac{\dx}{2}\phi_x-\frac{\dy}{2}\phi_y\right)\rho_{j,k}}{2},\\
&E_{j,k}=\frac{p_{j,k}}{\gamma-1}+\frac{\rho_{j,k}}{2}(u_{j,k}^2+v_{j,k}^2),\\
&Q_{\jmh,k}=Q_{\jph,k}-\dx\,\rho_{j,k}(\phi_x)_{j,k},\quad R_{j,\kph}=R_{j,\kmh}+\dy\,\rho_{j,k}(\phi_y)_{j,k}.
\end{aligned}
\end{equation*}

\bibliographystyle{siam}
\bibliography{ref_grav}
\end{document}